\documentclass[12pt]{amsart}
\usepackage{amssymb}
\usepackage{verbatim}

\begin{document}

\newtheorem{theorem}{Theorem}    
\newtheorem{proposition}[theorem]{Proposition}
\newtheorem{conjecture}[theorem]{Conjecture}
\def\theconjecture{\unskip}
\newtheorem{corollary}[theorem]{Corollary}
\newtheorem{lemma}[theorem]{Lemma}
\newtheorem{sublemma}[theorem]{Sublemma}
\newtheorem{observation}[theorem]{Observation}
\theoremstyle{definition}
\newtheorem{definition}{Definition}
\newtheorem{notation}[definition]{Notation}
\newtheorem{remark}[definition]{Remark}
\newtheorem{question}[definition]{Question}
\newtheorem{questions}[definition]{Questions}
\newtheorem{example}[definition]{Example}
\newtheorem{problem}[definition]{Problem}
\newtheorem{exercise}[definition]{Exercise}

\numberwithin{theorem}{section}
\numberwithin{definition}{section}
\numberwithin{equation}{section}

\def\earrow{{\mathbf e}}
\def\rarrow{{\mathbf r}}
\def\uarrow{{\mathbf u}}
\def\tpar{T_{\rm par}}
\def\apar{A_{\rm par}}

\def\reals{{\mathbb R}}
\def\torus{{\mathbb T}}
\def\heis{{\mathbb H}}
\def\integers{{\mathbb Z}}
\def\naturals{{\mathbb N}}
\def\complex{{\mathbb C}\/}
\def\distance{\operatorname{distance}\,}
\def\support{\operatorname{support}\,}
\def\dist{\operatorname{dist}\,}
\def\Span{\operatorname{span}\,}
\def\degree{\operatorname{degree}\,}
\def\kernel{\operatorname{kernel}\,}
\def\dim{\operatorname{dim}\,}
\def\codim{\operatorname{codim}}
\def\trace{\operatorname{trace\,}}
\def\Span{\operatorname{span}\,}
\def\ZZ{ {\mathbb Z} }
\def\p{\partial}
\def\rp{{ ^{-1} }}
\def\Re{\operatorname{Re\,} }
\def\Im{\operatorname{Im\,} }
\def\ov{\overline}
\def\eps{\varepsilon}
\def\lt{L^2}
\def\diver{\operatorname{div}}
\def\curl{\operatorname{curl}}
\def\etta{\eta}
\newcommand{\norm}[1]{ \|  #1 \|}
\def\Span{\operatorname{span}}
\def\expect{\mathbb E}
\def\paraboloid{{{\mathbb P}^2}}
\def\Sbest{{\mathbf S}}
\def\Pbest{{\mathbf P}}

\newcommand{\Norm}[1]{ \left\|  #1 \right\| }
\newcommand{\set}[1]{ \left\{ #1 \right\} }
\def\one{\mathbf 1}
\newcommand{\modulo}[2]{[#1]_{#2}}

\def\scriptf{{\mathcal F}}
\def\scriptg{{\mathcal G}}
\def\scriptm{{\mathcal M}}
\def\scriptb{{\mathcal B}}
\def\scriptc{{\mathcal C}}
\def\scriptt{{\mathcal T}}
\def\scripti{{\mathcal I}}
\def\scripte{{\mathcal E}}
\def\scriptv{{\mathcal V}}
\def\scriptS{{\mathcal S}}
\def\scripta{{\mathcal A}}
\def\scriptr{{\mathcal R}}
\def\scripto{{\mathcal O}}
\def\scripth{{\mathcal H}}
\def\frakv{{\mathfrak V}}

\author{Michael Christ}
\address{
        Michael Christ\\
        Department of Mathematics\\
        University of California \\
        Berkeley, CA 94720-3840, USA}
\email{mchrist@math.berkeley.edu}
\thanks{The first author was supported in part by NSF grant
DMS-0901569. The second author was supported by the National Science
Foundation under agreement DMS-0635607. Any opinions, findings, and conclusions
or recommendations expressed in this paper are those of the authors
and do not necessarily reflect the views of the National Science Foundation.}
\author{Shuanglin Shao}
\address{Shuanglin Shao\\
School of Mathematics, Institute for Advanced Study, Princeton, NJ 08540
\\
IMA, University of Minnesota, Minneapolis, MN 55455
}
\email{slshao@ima.umn.edu}

\date{
May 27, 2010.}

\title[Extremals for a Fourier restriction inequality]
{Existence of Extremals \\ for a Fourier Restriction Inequality}

\begin{abstract}
The adjoint Fourier restriction inequality
of Tomas and Stein states
that the mapping $f\mapsto \widehat{f\sigma}$
is bounded from $\lt(S^2)$ to $L^4(\reals^3)$.
We prove that there exist functions which extremize
this inequality, and that any extremizing
sequence of nonnegative functions has a subsequence
which converges to an extremizer.
\end{abstract}

\maketitle
{\Small \tableofcontents}

\section{Introduction}

Let $S^2$ denote the unit sphere in $\reals^3$, equipped
with surface measure $\sigma$. The adjoint Fourier restriction
inequality of Tomas and Stein, for $S^2$, states that
there exists $C<\infty$ such that
\begin{equation} \label{inequalityR}
\norm{\widehat{f\sigma}}_{L^4(\reals^3)}
\le C\norm{f}_{\lt(S^2,\sigma)}
\end{equation}
for all $f\in\lt(S^2)$.
With the Fourier transform defined to be
$\widehat{g}(\xi) = \int e^{-ix\cdot\xi}g(x)\,dx$,
denote by
\begin{equation}
\scriptr
=\sup_{0\ne f\in\lt(S^2)} \norm{\widehat{f\sigma}}_{L^4(\reals^3)}
\ \big/\ \norm{f}_{\lt(S^2,\sigma)}
\end{equation}
the optimal constant in the inequality \eqref{inequalityR}.

\begin{definition}
An extremizing sequence for the inequality \eqref{inequalityR}
is a sequence $\{f_\nu\}$ of functions in $\lt(S^2)$
satisfying $\norm{f_\nu}_2\le 1$,
such that $\norm{\widehat{f_\nu\sigma}}_{L^4(\reals^3)}\to\scriptr$
as $\nu\to\infty$.

An extremizer for the inequality \eqref{inequalityR}
is a function $f\ne 0$ which satisfies $\norm{\widehat{f\sigma}}_4=\scriptr\norm{f}_2$.
\end{definition}

The main result of this paper is:
\begin{theorem} \label{thm:main}
There exists an extremizer in $\lt(S^2)$ for the
inequality \eqref{inequalityR}.
\end{theorem}

The inequality dual to \eqref{inequalityR} is
$\norm{\widehat{h}}_{\lt(S^2,\sigma)}
\le C\norm{h}_{L^{4/3}(\reals^3)}$.
If $f$ extremizes \eqref{inequalityR},
then $\widehat{f\sigma}\cdot|\widehat{f\sigma}|^2$
extremizes the dual inequality.

\begin{definition}
A sequence of functions
in $\lt(S^2)$
is precompact if any subsequence has a sub-subsequence which is
Cauchy in $\lt(S^2)$.
\end{definition}

Nonnegative functions play a special role in our analysis, because
\begin{equation} \label{onlypositivefnsmatter}
\norm{\,\widehat{|f|\sigma}\,}_4\ge\norm{\widehat{f\sigma}}_4
\text{ for all $f\in\lt(S^2)$.}
\end{equation}
Therefore if $\{f_\nu\}$ is an extremizing sequence,
so is $\{|f_\nu|\}$.
Any limit, in the $\lt$ norm, of an extremizing sequence is of course an extremizer.
Thus the following implies Theorem~\ref{thm:main}.
\begin{theorem} \label{thm:preconvergence}
Any extremizing sequence of nonnegative functions in $\lt(S^2)$ for the inequality
\eqref{inequalityR} is precompact.
\end{theorem}


In particular, the set of all nonnegative extremizers is itself compact.
We do not know whether nonnegative extremizers are unique modulo rotations of $S^2$
and multiplication by constants.
They do possess the following symmetry, which will be useful in our analysis.

\begin{theorem}
Every extremizer satisfies $|f(-x)|=|f(x)|$ for almost every $x\in S^2$.
\end{theorem}
Proposition~\ref{prop:symmetrize} below states that more generally,
the quantity $\norm{\widehat{f\sigma}}_4$ never decreases under
$\lt$ norm preserving
symmetrization of $f$ with respect to the map $x\mapsto -x$.

For complex-valued extremizers and near-extremizers, the situation regarding
precompactness of extremizing sequences is different,
due to the presence of a noncompact group of symmetries of the inequality.
For $\xi\in\complex^3$ define $e_\xi(x) = e^{x\cdot\xi}$.
Then $\norm{\widehat{fe_{i\xi}\sigma}}_4=\norm{\widehat{f\sigma}}_4$
for arbitrary $\xi\in \reals^3$, $f\in\lt(S^2)$. Consequently complex-valued
extremizing sequences need not be precompact.
However, we will show in a sequel \cite{christshao2}
that this simple obstruction is the only one;
if $\{f_\nu\}$ is any complex-valued
extremizing sequence, then there exists a sequence
$\{\xi_\nu\}\subset\reals^3$
such that $e^{-ix\cdot\xi_\nu}f_\nu(x)$ is precompact.

The symmetries $f\mapsto f\cdot e^{ix\cdot\xi}$ merit further discussion. Matters
are clearer for the paraboloid
$\paraboloid=\{(y_1,y_2,y_3): y_3=\tfrac12y_1^2+\tfrac12 y_2^2\}$
than for $S^2$.
For $\paraboloid$, the analogues of these unimodular exponentials
are {\em quadratic} exponentials $e^{ix\cdot\eta + i\tau|x|^2}$
with $(\eta,\tau)\in\reals^{2+1}$; compare with $S^2$,
where $\xi\in\reals^3$ also ranges over a three-dimensional space.
To see the analogy,
consider a small neighborhood of $(0,0,1)\in S^2$,
equipped with coordinates $x'\in\reals^2$ so that $x=(x',(1-|x'|^2)^{1/2})$.
Then for $\xi=(0,0,\lambda)$, $e^{ix\cdot\xi} = \exp(i\lambda (1-\tfrac12|x'|^2
+O(|x'|^4))$ for small $x'$; thus for small $x'$ one has essentially quadratic
oscillation.
The presence of these symmetries among the extremizers for $\paraboloid$
implies that, in the language of concentration compactness theory \cite{kunze},
an extremizer $f$ can be tight at a scale $r$,
and $\widehat{f}$ can simultaneously be tight at a scale $\widehat{r}$,
with the product $r\cdot\widehat{r}$ arbitrarily large.

A routine variational argument leads to a generalized Euler-Lagrange equation.
Using Plancherel's Theorem, the connection between the Fourier transform and convolution,
and Cauchy-Schwarz, the definition of an extremizer can be reformulated.
\begin{proposition}
A function
$f\in\lt(S^2)$ is an extremizer if and only if
\begin{equation} \label{eulerlagrange}
\Big(f\sigma*f\sigma*f\sigma\Big)\Big|_{S^2} = {\mathbf S}^4 \norm{f}_2^2 f
\text{ a.e.\ on } S^2.
\end{equation}
\end{proposition}
Since the value of ${\mathbf S}$ has not been determined, this equation is
not entirely explicit.
By a routine variational argument,
any critical point $f$ of the functional $\norm{\widehat{f\sigma}}_4^4/\norm{f}_2^4$
satisfies the same equation, with $\Sbest$ replaced by some constant depending
on $f$; see for instance \cite{christquilodran},
where more general results of this type are justified.
\eqref{eulerlagrange} will be used in a forthcoming paper \cite{christshao2}
to prove that all critical points are infinitely differentiable.

Fundamental questions remain open, including:
\begin{questions}
Are extremizers unique modulo rotations and multiplication by constants?
Are constant functions extremizers?
\end{questions}
In this context, it is interesting to observe that
constant functions are {\em local} maxima.
Let $\one$ denote the constant function $f(x)\equiv 1$.
\begin{theorem} \label{thm:localmax}
There exists $\delta>0$ such that
whenever $\norm{f-\one}_{\lt(S^2)}<\delta$,
\begin{equation}
\frac{\norm{\widehat{f\sigma}}_4^4}{\norm{f}_2^4}
\le
\frac{\norm{\widehat{\sigma}}_4^4}{\norm{\one}_2^4},
\end{equation}
with equality only if $f$ is constant.
\end{theorem}

Let $\paraboloid$ be the paraboloid introduced above.
Let $\sigma_P$ be the measure $d\sigma_P= dx_1\,dx_2$
on $\paraboloid$.\footnote{See \cite{christextremal} for a brief discussion of the naturality
of this measure from a geometric perspective.}
Then the mapping $f\mapsto \widehat{f\sigma_P}$
is likewise bounded from $\lt(\paraboloid,\sigma_P)$ to $L^4(\reals^3)$.
Denote by $\scriptr_{\paraboloid}$ the optimal constant in the inequality
\begin{equation} \label{Rparaboloiddefn}
\norm{\widehat{f\sigma_P}}_{L^4(\reals^3)}
\le \scriptr_{\paraboloid}\norm{f}_{\lt(\paraboloid,\sigma_P)}.
\end{equation}
Foschi \cite{foschi} has proved that extremals exist for this inequality, and moreover,
that every Gaussian function of $(x_1,x_2)$ is an extremal; alternative
proofs were  given by
Hundertmark and Zharnitsky \cite{HZ}
and by Bennett, Bez, Carbery, and Hundertmark  \cite{hotstrichartz}.
The simple relation $\scriptr\ge\scriptr_{\paraboloid}$
is of significance for our discussion.
This relation follows from examination of a suitable sequence of trial functions
$f_\nu$, such that $f_\nu(x)^2\,dx$ converges weakly to a Dirac mass on $S^2$,
and $f_\nu$ is approximately a Gaussian in suitably rescaled coordinates,
depending on $\nu$. It is essential for this comparison that $\paraboloid$
has the same curvature as $S^2$, which explains the factors of $\tfrac12$
in the definition of $\paraboloid$.

The first author to discuss existence of extremizers for Strichartz/Fourier
restriction inequalities was apparently Kunze \cite{kunze}, who proved the
existence of extremizers for the
parabola in $\reals^2$, and showed that (in our notation) any nonnegative
extremizing sequence is precompact.
Several papers have subsequently dealt with related problems,
in some cases
determining all extremizers explicitly
\cite{foschi}, \cite{HZ}, \cite{hotstrichartz}, \cite{carneiro},
in other cases merely proving existence
\cite{shao}.
A powerful result which leads easily \cite{shao} to existence of
extremizers is the profile decomposition; see \cite{begoutvargas}.
Of these works, the one most closely related to ours is that of Kunze.
One difficulty which we face is the lack of exact scaling symmetries.
In some facets of the analysis this is merely a technical obstacle,
but it is bound up with the most essential obstacle,
which is the possibility that the optimal constant might be achieved only in
a limit where $|f|^2$ tends to a Dirac mass, or a sum of two Dirac masses.

Our analysis follows the general concentration compactness framework
developed by Lions \cite{lions1984a},\cite{lions1984b},\cite{lions1985a},\cite{lions1985b}.
We have elected to make the exposition self-contained in this respect,
not drawing on that theory; to do so would apparently not dramatically shorten
the exposition, since most of our labor is devoted to specific issues
raised by the character of a particular nonlocal operator.

Existence of extremals for a convolution inequality in which curvature
plays an essential role, as it does here,
was proved in \cite{christextremal}.
The underlying geometry governing \cite{christextremal} is more
subtle, but the operator analyzed there is merely linear,
while the analysis of the present paper is bilinear.
Yet despite differences in details,
that analysis and the method of the present paper have much in common.
The role of an inequality of Moyua, Vargas, and Vega \cite{mvv}
used here was played in \cite{christextremal} by \cite{quasiextremal}.

\medskip
We are indebted to Terence~Tao for bringing the question to our attention,
and to Diogo Oliveira e Silva for useful comments on the exposition.

\section{Outline of the proof and definitions}
The following overview of the proof includes notations, definitions,
and statements of intermediate results which are not repeated subsequently,
and thus is an integral part of the presentation.

\paragraph{\bf Step 1.}
The first step is quite simple, but in it a critical distinction appears
between our problem for $S^2$, and for higher-dimensional spheres.
The inequality
$\norm{\widehat{f\sigma}}_{L^4(\reals^3)}
\le C\norm{f}_{\lt(S^2,\sigma)}$
is equivalent, by Plancherel's theorem, to
\begin{equation} \label{secondSversion}
\norm{f\sigma*f\sigma}_{\lt(\reals^3)} \le \Sbest^2\norm{f}_{\lt(S^2)}^2,
\end{equation}
where
\begin{equation}
\scriptr= (2\pi)^{3/4}\Sbest
\end{equation}
and $*$ denotes convolution of measures.
This has been exploited in
\cite{kunze},\cite{foschi},\cite{HZ},\cite{hotstrichartz}.
In higher dimensions, the exponent $4$ is replaced by
an exponent which is no longer an even integer, and no such equivalence is available.

Now the pointwise inequality
$|f\sigma*f\sigma|\le |f|\sigma*|f|\sigma$,
the relation $\widehat{\mu*\nu}=\widehat{\mu}\widehat{\nu}$,
and Plancherel's theorem imply
\begin{lemma}
For any complex-valued function $f\in\lt(S^2)$,
\begin{equation}
\norm{\widehat{f\sigma}}_{L^4(\reals^3)}
\le
\norm{\widehat{|f|\sigma}}_{L^4(\reals^3)}.
\end{equation}
Therefore if $f$ is an extremizer for inequality \eqref{inequalityR},
then so is $|f|$;
if $\{f_\nu\}$ is an extremizing sequence, so is $\{|f_\nu|\}$.
\end{lemma}
This permits us to work with nonnegative functions throughout the
analysis. For much of our analysis this makes no difference,
but nonnegativity will be useful in Step 7,
allowing an elementary approach to a step whose analogue
in higher dimensions seems to require more sophisticated techniques.

\paragraph{\bf Step 2.}
A potential obstruction to the existence of extremizers, and certainly to
the precompactness of arbitrary extremizing sequences, is the possibility
that for an extremizing sequence satisfying $\norm{f_\nu}_2=1$,
$|f_\nu|^2$ could conceivably converge weakly to a Dirac mass at a point of $S^2$.
Indeed, if $\scriptr$ were to equal $\scriptr_\paraboloid$, then there would
necessarily
exist extremizing sequences of this type. Therefore an essential step
in our analysis is to prove that $\scriptr>\scriptr_\paraboloid$.

In fact, as will be explained below, this is true in two distinct ways.
The more superficial is this:
\begin{lemma} \label{lemma:32}
Let $g\in\lt(S^2)$ be supported in $\{x\in S^2: x_3>\tfrac12\}$.
Define $f(x)=2^{-1/2}g(x)+ 2^{-1/2}\overline{g(-x)}$.
Then $\norm{f}_2=\norm{g}_2$, and
\begin{equation}
\norm{f\sigma*f\sigma}_{\lt(\reals^3)}
= (3/2)^{1/2}
\norm{g\sigma*g\sigma}_{\lt(\reals^3)}.
\end{equation}
\end{lemma}
Define the optimal constant in the corresponding
inequality for the paraboloid to be
\begin{equation} \label{Pbestdefn}
\Pbest =
\sup_{0\ne g\in\lt(\paraboloid,\sigma_P)}
\frac{
\norm{g\sigma_P*g\sigma_P}_{L^2(\reals^3)}^{1/2} }
{ \norm{g}_{\lt(\paraboloid,\sigma_P)} }\,\,.
\end{equation}
Thus the optimal constants for the sphere and paraboloid satisfy
\begin{corollary} \label{cor:greater}
\begin{equation}
\Sbest\ge (3/2)^{1/4}\Pbest.
\end{equation}
\end{corollary}

\paragraph{\bf Step 3.}
Step 2 leaves open many possibilities, the simplest of which is
that an extremizing sequence
might concentrate at a pair of antipodal points,
that is, $|f_\nu|^2$ might converge weakly
to a linear combination of two Dirac masses, at antipodal points
$z,-z$.
We will see that this scenario is the crux of the problem.
The crucial ingredient in excluding it
is an inequality $\Sbest > (3/2)^{1/4}\Pbest$.
We will give two independent proofs of this inequality. The first gives
a precise improvement:

\begin{lemma} \label{lemma:constantfunction}
\begin{equation}
\Sbest\ge 2^{1/4}\Pbest.
\end{equation}
\end{lemma}
Equivalently, $\scriptr\ge 2^{1/4}\scriptr_\paraboloid$.
This is proved by exact computation of $\norm{f\sigma*f\sigma}_2$
for $f\equiv 1$.
We do not know whether constant functions are in fact extremal for
\eqref{inequalityR}, or equivalently, whether $\Sbest = 2^{1/4}\Pbest$.
Constants are indeed critical points of the associated functional,
and thus satisfy a (possibly) modified Euler-Lagrange equation
\eqref{eulerlagrange}, in which $\Sbest$ is replaced by $2^{1/4}\Pbest$.

An alternative proof that $\Sbest>(3/2)^{1/4}\Pbest$,
along perturbative lines, is given in \S\ref{section:calculation}.

\paragraph{\bf Step 4.}

\begin{definition}
A complex-valued function $f\in\lt(S^2)$
is said to be even if $f(-x)=\overline{f(x)}$
for almost every $x\in S^2$.
\end{definition}
We will be working almost exclusively with nonnegative
functions, for which this condition becomes
$f(-x)\equiv f(x)$.

\begin{definition}
Let $f\in\lt(S^2)$ be nonnegative.
The antipodally symmetric rearrangement $f_\star$
is the unique nonnegative element of $\lt(S^2)$ which satisfies
\begin{alignat}{2}
f_\star(-x)&=f_\star(x)
&&\text{ for all } x\in S^2,
\\
f_\star(x)^2 + f_\star(-x)^2 &= f(x)^2 + f(-x)^2
&&\text{ for all } x\in S^2.
\end{alignat}
\end{definition}

\begin{proposition} \label{prop:symmetrize}
For any nonnegative $f\in\lt(S^2)$,
\begin{equation}
\norm{f\sigma*f\sigma}_{\lt(\reals^3)}
\le
\norm{f_\star\,\sigma*f_\star\,\sigma}_{\lt(\reals^3)},
\end{equation}
with strict inequality unless $f=f_\star$ almost everywhere.
Consequently any extremizer for the inequality \eqref{inequalityR} satisfies
$|f(-x)|=|f(x)|$ for almost every $x\in S^2$.
\end{proposition}
\noindent
An equivalent formulation is that
$\norm{\widehat{f\sigma}}_4 \le \norm{\widehat{f_\star\,\sigma}}_4$.

This allows us to restrict attention from nonnegative functions
to even nonnegative functions throughout the discussion.
This simplification is more convenient than essential.

\paragraph{\bf Step 5.}
A first key step towards gaining control of near-extremals
has already been essentially accomplished by Moyua, Vargas, and Vega \cite{mvv}.

\begin{definition}
The cap $\scriptc=\scriptc(z,r)$ with center $z\in S^2$
and radius $r\in(0,1]$ is
the set of all points $y\in S^2$
which lie in the same hemisphere, centered at $z$, as $z$ itself,
and which satisfy $|\pi_{H_z}(y)|<r$,
where the subspace $H_z\subset\reals^3$ is the orthogonal complement of $z$
and $\pi_{H_z}$ denotes the orthogonal projection onto $H_z$.
\end{definition}

\begin{lemma} \label{lemma:mvv}
For any $\delta>0$ there exist
$C_\delta<\infty$ and $\eta_\delta>0$
with the following property.
If $f\in\lt(S^2)$ satisfies
$\norm{f\sigma*f\sigma}_2
\ge \delta^2 \Sbest^2\norm{f}_2^2$
then there exist a decomposition
$f=g+h$ and a cap $\scriptc$
satisfying
\begin{align}
&0\le |g|,|h|\le |f|,
\\
&g,h \text{ have disjoint supports},
\\
&|g(x)|\le C_\delta \norm{f}_2|\scriptc|^{-1/2}\chi_\scriptc(x)\ \forall x,
\\
&\norm{g}_2\ge\eta_\delta\norm{f}_2.
\end{align}
\end{lemma}
The first conclusion is of course redundant. If $f\ge 0$ then it follows that $g,h\ge 0$
almost everywhere.

Lemma~\ref{lemma:mvv} is a corollary of Theorem~4.2 of \cite{mvv}.
It can also be proved via arguments closely related to those in
\cite{quasiextremal}.

\paragraph{\bf Step 6.}
This step is related to the techniques used in \cite{christextremal}.

\begin{definition}
Let $\scriptb\subset\reals^2$ denote the unit ball.
To any cap of radius $\le 1$
is associated a rescaling map $\phi_\scriptc:\scriptb\leftrightarrow\scriptc$.
For $z=(0,0,1)$,
$\phi_\scriptc(y_1,y_2) = (ry_1,ry_2,(1-r^2|y|^2)^{1/2})$.
For general $z$,
define
$\psi_z(y)= r^{-1}L(\pi(y))$
where $\pi$ is again the orthogonal projection onto $H_z$,
$L:H_z\leftrightarrow \reals^2$ is an arbitrary linear isometry,
and $\phi_{\scriptc(z,r)}=\psi^{-1}$.
For small $r>0$, $\phi_{\scriptc(z,r)}$ is naturally defined on
$B(0,r^{-1})$, which it maps into a cap of radius $1$ in $S^2$.
\end{definition}

\begin{definition}
Define the pullbacks
\begin{equation}
\phi_\scriptc^*f(y) = r\cdot (f\circ\phi_\scriptc)(y)
\end{equation}
where $r$ is the radius of the cap $\scriptc$.
\end{definition}
This definition makes sense provided that $f$ is supported
in the cap of radius $1$ concentric with $\scriptc$.
These pullbacks preserve norms up to uniformly bounded factors
provided that $r\le r_0<1$;
$\norm{\phi_\scriptc^* f}_{\lt(\reals^2)}\asymp
\norm{f}_{\lt(S^2,\sigma)}$
with the ratio of these norms is bounded above and below by positive,
finite constants, uniformly in $f,r,z$.
For the sake of definiteness we set $r_0=\tfrac12$.

\begin{definition} \label{defn:normalized}
Let $\Theta:[1,\infty)\to(0,\infty)$
satisfy $\Theta(R)\to 0$ as $R\to\infty$.
A function $f\in\lt(S^2)$ is said to be upper normalized, with
gauge function $\Theta$, with respect to a cap
$\scriptc=\scriptc(z,r)\subset S^2$
of radius $r$ and center $z$ if
\begin{align}
\norm{f}_2&\le C<\infty,
\\
\label{eq:normalization1}
\int_{|f(x)|\ge Rr^{-1}}|f^2(x)|\,dx &\le \Theta(R)
\qquad\forall R\ge 1,
\\
\label{eq:normalization2}
\int_{|x-z|\ge Rr} |f^2(x)|\,dx &\le \Theta(R)
\qquad\forall R\ge 1.
\end{align}
An even function $f$ is said to be upper even-normalized
with respect to $\Theta,\scriptc$
if $f$ can be decomposed as $f=f_+ + f_-$ where $f_-(x)\equiv \overline{f_+(-x)}$,
and
$f_+$ is upper normalized with respect to $\Theta,\scriptc$.

A function $f\in\lt(\reals^2)$
is said to be upper normalized with respect to the unit ball in $\reals^2$
if
$\norm{f}_2\le C<\infty$,
$\int_{|f(x)|\ge R}|f^2(x)|\,dx \le \Theta(R)$
for all $R\ge 1$,
and
$\int_{|x|\ge R} |f^2(x)|\,dx \le \Theta(R)$
for all $R\ge 1$.
\end{definition}
We will usually omit the phrase
``with gauge function $\Theta$'', and will say that a function
is upper normalized if it satisfies the required inequalities
with respect to some appropriate function $\Theta$ which
has been, in principle, specified earlier in the discussion.

\begin{definition}
A nonzero function $f\in\lt(S^2)$ is said to be $\delta$--nearly extremal
for the inequality \eqref{secondSversion}
if
\begin{equation}\label{DNE}
\norm{f\sigma*f\sigma}_{\lt(\reals^3)}\ge (1-\delta)^2\Sbest^2\norm{f}_2^2.
\end{equation}
\end{definition}

\begin{proposition} \label{prop:normalization}
There exists a function $\Theta:[1,\infty)\to(0,\infty)$
satisfying $\Theta(R)\to 0$ as $R\to\infty$ with the following property.
For any $\eps>0$ there exists $\delta>0$
such that any nonnegative even function $f\in\lt(S^2)$ satisfying
$\norm{f}_2=1$ which is $\delta$--nearly extremal
may be decomposed as $f=F+G$ where
$F,G$ are even and nonnegative with disjoint supports,
$\norm{G}_2<\eps$,
and there exists a cap $\scriptc$
such that $F$ is upper even-normalized with respect to $\scriptc$.
\end{proposition}

The proof is a largely formal argument which rests on two inputs:
Lemma~\ref{lemma:mvv}, and the observation that for two caps $\scriptc,\scriptc'$,
$\norm{\chi_\scriptc\sigma*\chi_{\scriptc'}\sigma}_2\ll
|\scriptc|^{1/2}|\scriptc'|^{1/2}$
unless $\scriptc,\scriptc'$ have comparable radii and nearby centers.

\paragraph{\bf Step 7.}
This is the sole step which works only for nonnegative extremizing sequences.
It is also the most involved step of the argument.

\begin{proposition} \label{prop:precompactness}
Let $\{f_\nu\}\subset\lt(S^2)$ be an extremizing sequence
of nonnegative even functions for the inequality
\eqref{secondSversion}, satisfying $\norm{f_\nu}_2\equiv 1$.
Suppose that each $f_\nu$ is upper even-normalized with respect
to a cap $\scriptc_\nu=\scriptc(z_\nu,r_\nu)$, with constants uniform in $\nu$.
Then for any $\eps>0$ there exists $C_\eps<\infty$ with the following property.
For every $\nu$, if $r_\nu\le \tfrac12$ then
$\phi_\nu^*(f_\nu)$ may be decomposed
as $\phi_\nu^*(f_\nu)=G_\nu+H_\nu$ where
\begin{gather}
\norm{H_\nu}_2<\eps,
\\
G_\nu \text{ is supported where $|x|\le C_\eps$,}
\\
\norm{G_\nu}_{C^1}\le C_\eps.
\end{gather}
If $r_\nu\ge \tfrac12$
then $f_\nu$ itself  may be decomposed as
$f_\nu=g_\nu+h_\nu$ where
$\norm{h_\nu}_2<\eps$
and
$\norm{g_\nu}_{C^1}\le C_\eps$.
\end{proposition}
\noindent Here $\phi_\nu^* = \phi_{\scriptc_\nu}^*$.

The idea is that if $g\in\lt(\reals^2)$ satisfies $\norm{g}_2\sim 1$,
if $g$ is upper normalized with respect to the unit ball,
and if $g$ is nonnegative, then $\int_{|\xi|\lesssim 1}|\widehat{g}(\xi)|^2\,d\xi$
is bounded below by a universal strictly positive constant.
If precompactness were to fail, then $g_\nu=\phi_\nu^*(f_\nu)$ would
have to satisfy
$\int_{|\xi|\ge \Lambda_\nu}|\widehat{g_\nu}(\xi)|^2\,d\xi\ge\eta>0$,
with $\limsup\Lambda_\nu=\infty$.
Thus in an appropriately rescaled sense, $f_\nu$ is a superposition of a slowly
varying part, plus a highly oscillatory part, with perhaps some intermediate portion.
For the bilinear expression $f\sigma*f\sigma$, we show that the cross term resulting
from the high and low frequency parts is small, and that this contradicts extremality.

An application of Rellich's lemma then yields:
\begin{corollary} \label{cor:rellich}
Let $\{f_\nu\}\subset\lt(S^2)$ be an extremizing
sequence of even nonnegative functions for the inequality \eqref{secondSversion},
which are upper even-normalized with respect to
a sequence of caps $\{\scriptc_\nu=\scriptc(z_\nu,r_\nu)\}$.
\newline
(i)
If $r_\nu\to 0$ then
$\{\phi_\nu^*(f_\nu)\} \text{ is precompact in } \lt(\reals^2)$.
\newline
(ii)
If $\liminf_{\nu\to\infty} r_\nu>0$ then
$\{f_\nu\} \text{ is precompact in } \lt(S^2)$.
\end{corollary}

\paragraph{\bf Step 8.}

\begin{proposition} \label{prop:notsmall}
Let $\{f_\nu\}$ be as in Proposition~\ref{prop:precompactness}.
Then $\liminf_{\nu\to\infty} r_\nu>0$.
\end{proposition}

The proof of Proposition~\ref{prop:notsmall} proceeds by contradiction.
If $\{f_\nu\}$ satisfies $r_\nu\to 0$, then a rescaling and transference
argument can be used to a corresponding sequence of functions $\{\tilde f_\nu\}$
on $\paraboloid$, which is precompact in $\lt(\paraboloid)$.
In coordinates rescaled according to $r_\nu$, each $\tilde f_\nu$ is acted upon by an
adjoint Fourier restriction operator associated to a hypersurface
which depends on $r_\nu$, and which approaches $\paraboloid$ as $r_\nu\to 0$.
The precompactness of $\{\tilde f_\nu\}$ and convergence of these
hypersurfaces can be used to obtain a limit $F\in\lt(\paraboloid)$
which satisfies $\norm{\widehat{F\sigma_P}}_4/\norm{F}_2
= (3/2)^{-1/4}\lim_{\nu\to\infty} \norm{\widehat{f_\nu\sigma }}_4/\norm{f_\nu}_2$.
It follows that $\scriptr_\paraboloid\ge (3/2)^{-1/4}\scriptr$.
But this contradicts the inequality $\scriptr\ge 2^{1/4}\scriptr_\paraboloid$ of Step 3.

\paragraph{\bf Conclusion.}
Extremizing sequences exist. We have shown that there exists an extremizing
sequence which consists of even, nonnegative functions. Such a sequence
is upper even-normalized with respect to a sequence of caps.
By Proposition~\ref{prop:notsmall},
the radii of these caps cannot tend to zero.
By Corollary~\ref{cor:rellich}, such a sequence has a
subsequence which converges in $\lt(S^2)$. The limit
of such a subsequence is obviously an extremal.

\paragraph{\bf Not a Step.}
As explained above in Step 2,
the fundamental potential obstruction to the
precompactness of (nonnegative) extremizing sequences
was the possibility that
$|f_\nu|^2$
could converge weakly to a Dirac mass,
or to a sum of two Dirac masses at a pair of antipodal points.
The following result examines a natural one-parameter
family of candidate trial functions.

\begin{proposition} \label{prop:perturbative}
For all $\xi\in\reals^3$ with $|\xi|$ sufficiently large,
\begin{equation}
\norm{\widehat{e_\xi\sigma}}_{L^4(\reals^3)}>\scriptr_{\paraboloid}\norm{e_\xi}_{\lt(S^2)}.
\end{equation}
\end{proposition}
When $\xi=(0,0,\lambda)$,
$e_\xi^2/\norm{e_\xi}_2^2$ does converge weakly  as $\lambda\to+\infty$
to a constant multiple of a Dirac mass at $(0,0,1)$.
Proposition~\ref{prop:perturbative} is proved
in \S\ref{section:calculation} via a perturbative calculation.

By taking the considerations of Step 2 involving even functions into account,
Proposition~\ref{prop:perturbative} provides an alternative route to
the essential comparison $\Sbest> (3/2)^{1/4}\Pbest$.
Although Proposition~\ref{prop:perturbative}
is not strictly necessary for the main lines of our proof,
the calculation which underlies it
will be useful in a generalization to manifolds other than $S^2$,
and it is reassuring to be freed of complete reliance on the
calculation, carried out in  Lemma~\ref{lemma:constantfunction},
of a single real number.

\section{Step $2$: $\Sbest\ge (3/2)^{1/4}\Pbest$}

Let $\tilde f(x)=f(-x)$.
Denote by $\langle F,G\rangle$
the pairing of two functions in $\lt(\reals^3)$.
\begin{lemma} \label{lemma:switchconvolutionfactors}
For any four real-valued functions $f_j\in\lt(S^2)$,
\begin{equation} \label{movingfactorsaround}
\langle f_1\sigma*f_2\sigma,\,f_3\sigma*f_4\sigma\rangle
=
\langle f_1\sigma*\tilde f_3\sigma,\,\tilde f_2\sigma*f_4\sigma\rangle
\end{equation}
and
\begin{equation} \label{sameL2norms}
\norm{f_1\sigma*f_2\sigma}_{\lt(\reals^3)}
= \norm{f_1\sigma*\tilde f_2\sigma}_{\lt(\reals^3)}
\end{equation}
\end{lemma}

\begin{proof}
The inequality $\norm{f\sigma*g\sigma}_{\lt(\reals^3)}
\le \Sbest^2\norm{f}_{\lt(\sigma)}\norm{g}_{\lt(\sigma)}$
ensures that these quantities are well-defined,
and that the first identity holds for all $\lt$
functions provided that it holds for all nonnegative continuous functions $f_j$.
In that case $f_3\sigma*f_4\sigma(x)\le C|x|^{-1}$
for all $x\in\reals^3$, where $C<\infty$
depends on $f_3,f_4$, and $f_3\sigma*f_4\sigma$ is continuous
except at $x=0$.
For any $F\in C^0(\reals^3)$ and $f_j\in C^0(S^2)$,
\[\langle f_1\sigma*f_2\sigma,\,F\rangle
= \int (\tilde f_2\sigma*F)f_1\,d\sigma,\]
a consequence of the definition of convolution of measures
and Fubini's theorem.
Limiting arguments then lead to \eqref{movingfactorsaround}.

\eqref{sameL2norms} now follows:
\begin{multline*}
\norm{f_1\sigma*f_2\sigma}_{\lt(\reals^d)}^2
=\langle f_1\sigma*f_2\sigma,\,f_1\sigma*f_2\sigma\rangle
= \langle f_1\sigma*f_2\sigma,\,f_2\sigma*f_1\sigma\rangle
\\
= \langle f_1\sigma*\tilde f_2\sigma,\,\tilde f_2\sigma*f_1\sigma\rangle
= \langle f_1\sigma*\tilde f_2\sigma,\,f_1\sigma*\tilde f_2\sigma \rangle
= \norm{f_1\sigma*\tilde f_2\sigma}_{\lt}^2.
\end{multline*}
\end{proof}

\begin{proof}[Proof of Lemma~\ref{lemma:32}]
Let $\eps>0$. Choose $f\in \lt(S^2)$,
supported in an open hemisphere, satisfying
$\norm{f\sigma*f\sigma}_2^2\ge (\Pbest-\eps)^2\norm{f}_{\lt(S^2)}^2$.
By replacing $f$ by $|f|$, we may assume that $f\ge 0$.

Set $d\mu = f\,d\sigma$.
Let $g(x)=f(x)+f(-x)$ and $d\nu = g\,d\sigma = \mu+\tilde\mu$.
The two terms $f(x)$ and
$f(-x)$ have disjoint supports,
so
\[
\norm{g}_{\lt(S^2)}^2 = 2\norm{f}_{\lt(S^2)}^2.
\]

Now
\[\nu*\nu = (\mu+\tilde\mu)*(\mu+\tilde\mu)
=(\mu*\mu) + (\tilde\mu*\tilde\mu) + 2(\mu*\tilde\mu).\]
Therefore
\[
\norm{\nu*\nu}_{\lt(\reals^3)}^2
\ge  \norm{\mu*\mu}_{\lt}^2
+ \norm{\tilde\mu*\tilde\mu}_{\lt}^2
+ 4 \norm{\mu*\tilde\mu}_{\lt}^2.
\]
Now
$\norm{\mu*\tilde\mu}_{\lt}^2
= \norm{\mu*\mu}_{\lt}^2$, as shown above.
Thus we find that
\[
\norm{\nu*\nu}_{\lt(\reals^3)}^2
\ge  6\norm{\mu*\mu}_{\lt}^2,
\]
while
\[
\norm{\nu}_{\lt(S^2)}^2 = 2\norm{\mu}_{\lt(S^2)}^2.
\]
Squaring the last identity we find a ratio $\tfrac64=\tfrac32$. Thus
$\Sbest^4 \ge \tfrac32  \Pbest^4$.
\end{proof}

\section{Step 3: $\Sbest\ge 2^{1/4}\Pbest$}


\begin{proof}[Proof of Lemma~\ref{lemma:constantfunction}]
We will obtain a lower bound for $\Sbest$ by calculating
$\norm{f\sigma*f\sigma}_2^2$ for $f\equiv 1$.

Recall certain facts:
The unit ball in $\reals^3$ has volume $4\pi/3$:
expressing this as the volume within the region
$|x_3|^2\le 1-|x'|^2$ gives
\[
\int_{|x'|\le 1} 2(1-|x'|^2)^{1/2}\,dx'
= 2\int_0^1 2\pi (1-r^2)^{1/2}r\,dr.
\]
The derivative of $(1-r^2)^{3/2}$
is $-3r(1-r^2)^{1/2}$,
and $(1-r^2)^{3/2}\big|_0^1=-1$.

Therefore
\[\sigma(S^2)=\frac{d}{dr}\tfrac43 \pi r^3\big|_{r=1}=4\pi,\]
and the volume form in $\reals^3$ in polar coordinates is
\[r^2\,dr\,d\sigma(\theta).\]

One calculates that
\[\sigma*\sigma(x) = a|x|^{-1}\chi_{|x|\le 2}\]
for a certain constant $a>0$. We will not need to evaluate $a$
(because it will cancel out at the very end of the calculation).
What we do need to know is that if we denote by $\mu$
the measure $dx'$ on the paraboloid $P=\{x\in\reals^3: x_3=\tfrac12|x'|^2\}$,
then
\[\mu*\mu(z) \equiv \tfrac{a}2\chi_\Omega\]
 where $\Omega$ denotes the support of $\mu*\mu$.
This factor of $\tfrac12$ in the definition of $P$ is required
to make the curvature of $P$ equal to the curvature of $S^2$;
one sees that they are equal by writing the equation for $S^2$
near the north pole as $x_3-1 = (1-|x'|^2)^{1/2}-1$
and Taylor expanding the right-hand side.
Note that the factor $a/2$ in the formula for $\mu*\mu$
agrees with the limit as $|x|\to 2$ of the function
$a/|x|$ which appears in the formula for $\sigma*\sigma$.
This asymptotic equality must hold since the two surfaces
have equal curvatures, hence the two convolutions must
agree on the diagonal of the maps $(x,y)\mapsto x+y$.
We will not prove that $\mu*\mu$ is constant on its support;
this is a reflection of the symmetry of the paraboloid
(including appropriate dilation symmetry)
and invariance of curvature under mappings of the
form $(x',x_3)\mapsto (x',x_3-L(x'))$
where $L:\reals^2\to\reals^1$ is linear.

The support of $\mu*\mu$ is
\[
\Omega=\{z: z_3>\tfrac14|z'|^2\}.
\]

It is known \cite{foschi},\cite{HZ}
that any Gaussian is
an extremizer for the paraboloid, and conversely.
Another proof that Gaussians extremize the inequality is in
\cite{hotstrichartz}.
Set
$F(x',x_3) = e^{-|x'|^2/2}\equiv e^{-x_3}$ on the paraboloid.
Observe that if $x+y=z\in\reals^3$,
then \[F(x)F(y) = e^{-x_3-y_3}=e^{-z_3}.\]
Therefore
\[
(F\mu*F\mu)(z) = \tfrac{a}2 e^{-z_3}\chi_{z_3>|z'|^2/4}.
\]
Consequently
\begin{multline*}
\norm{F\mu*F\mu}_2^2
=
\tfrac{a^2}4
\int_{z'\in\reals^2}\int_{z_3>|z'|^2/4}e^{-2z_3}\,dz
\\
=
\tfrac{a^2}4
\int_0^\infty 2\pi \int_{r^2/4}^\infty e^{-2s}\,ds\,r\,dr
\\
=
\tfrac{a^2}4
2\pi\int_0^\infty \tfrac12 e^{-r^2/2}\,r\,dr
=
\frac{\pi a^2}{4}.
\end{multline*}

On the other hand,
\begin{equation*}
\norm{\sigma*\sigma}_{\lt(\reals^3)}^2
=
\int_{|x|\le 2}a^2|x|^{-2}\,dx
= a^2 \int_0^2 r^{-2}\,4\pi r^2\,dr
= 4\pi a^2\int_0^2\,dr
= 8\pi a^2.
\end{equation*}

Meanwhile
\[
\norm{1}_{\lt(\sigma)}^2
= \sigma(S^2)=4\pi,
\]
and
\begin{equation*}
\norm{F}_{\lt(\mu)}^2
= \int_{\reals^2} e^{-2|x|^2/2} \,dx
=
\int_0^\infty e^{-r^2}2\pi r\,dr
= \pi.
\end{equation*}

Putting this all together,
\[
\frac{\norm{F\mu*F\mu}_2^2}
{\norm{F}_{\lt(\mu)}^4}
= \frac{a^2\pi/4}{\pi^2}
= \frac{a^2}{4\pi},
\]
while
\[
\frac{\norm{1\sigma*1\sigma}_2^2}
{\norm{1}_{\lt(\sigma)}^4}
= \frac{8\pi a^2}{(4\pi)^2}
= \frac{a^2}{2\pi}.
\]
The second ratio is equal to twice the first, as claimed.
\end{proof}

\section{Step 4: Symmetrization}

Proposition~\ref{prop:symmetrize} stated that for any dimension $d$,
$\norm{f\sigma*f\sigma}_{\lt(\reals^d)}
\le\norm{f_\star\,\sigma*f_\star\,\sigma}_{\lt(\reals^d)}$
for any nonnegative function $f\in\lt(S^{d-1})$,
where $f_\star$ denotes the even symmetrization of $f$.

\begin{proof}[Proof of Proposition~\ref{prop:symmetrize}]
Let $\sigma$ denote surface measure on $S^{d-1}$.
For $h\ge 0$,
\begin{equation} \label{fourfold}
\norm{h\sigma*h\sigma}_{\lt}^2
= \int h(a)h(b)h(c)h(d)\,d\lambda(a,b,c,d)
\end{equation}
for a certain nonnegative measure $\lambda$
which is supported on the set where $a+b=c+d$,
and which is invariant under the transformations
\begin{equation} \label{orbit}
\begin{aligned}
&(a,b,c,d)\mapsto (b,a,c,d),
\\
&(a,b,c,d)\mapsto (c,d,a,b),
\\
&(a,b,c,d)\mapsto (a,-c,-b,d)
\\
&(a,b,c,d)\mapsto (-a,-b,-c,-d).
\end{aligned}
\end{equation}
This invariance, which is essential to the discussion,
follows from the identities
\begin{align*}
f\sigma*g\sigma&=g\sigma*f\sigma,
\\
\langle f\sigma*g\sigma,h\sigma*k\sigma\rangle
&=
\langle h\sigma*k\sigma,f\sigma*g\sigma\rangle,
\\
\langle f\sigma*g\sigma,h\sigma*k\sigma\rangle
&= \langle f\sigma*\tilde h\sigma,\tilde g\sigma*k\sigma\rangle
\end{align*}
for arbitrary real-valued functions,
where $\tilde F(x)=F(-x)$.

Denote by $G$ the finite group of symmetries of $(\reals^d)^4$
which these generate.
$G$ has cardinality $48$. For exactly one of $a,-a$ appears;
suppose that $a$ appears.
There are $4$ places in which it can go.
$\pm b$ can then go into any of $3$ slots, but whether it is $+b$ or $-b$
is determined by which slot. There remain two slots into which $\pm c$ can go;
again, the $\pm$ sign is determined by the slot. $\pm d$ then goes into
the remaining slot, with the $\pm$ sign again determined.
The analysis is parallel if $-a$ appears.
Thus there are $2\times 4\times 3\times 2=48$ possibilities.

By the orbit of a point we mean its image under $G$;
by a generic point we mean one whose orbit has cardinality $48$.
In \eqref{fourfold}, it suffices to integrate only over all
{\em generic} $4$-tuples $(a,b,c,d)$ satisfying $a+b=c+d$,
since these form a set of full $\lambda$-measure.

To the orbit $\scripto$
we associate the functions
\begin{align*}
\scriptf(\scripto) &= \sum_{(a,b,c,d)\in\scripto}
f(a)f(b)f(c)f(d)
\\
\scriptf_\star(\scripto) &= \sum_{(a,b,c,d)\in\scripto}
f_\star(a)f_\star(b)f_\star(c)f_\star(d).
\end{align*}
Let $\Omega$ denote the set of all orbits of generic points.
We can write
\begin{align*}
\norm{f*f}_{\lt}^2 &= \int_{\Omega} \scriptf(\scripto)\
\,d\tilde\lambda(\scripto)
\\
\norm{f_\star*f_\star}_{\lt}^2 &= \int_{\Omega} \scriptf_\star(\scripto)
\,d\tilde\lambda(\scripto)
\end{align*}
for a certain nonnegative measure $\tilde\lambda$.
Therefore it suffices to prove that for any generic orbit $\scripto$,
\begin{equation} \label{oneorbitinequality}
\sum_{(a,b,c,d)\in\scripto} f(a)f(b)f(c)f(d)
\le
\sum_{(a,b,c,d)\in\scripto} f_\star(a)f_\star(b)f_\star(c)f_\star(d).
\end{equation}

Fix any generic ordered $4$-tuple $(a,b,c,d)$ satisfying $a+b=c+d$.
We prove \eqref{oneorbitinequality} for its orbit.
By homogeneity, it is no loss of generality to
assume that $f^2(a)+f^2(-a)=1$ and that the same holds
simultaneously for $b,c,d$.
Thus we may write
\[
f(a)=\cos(\varphi),
f(b)=\cos(\psi),
f(c)=\cos(\alpha),
f(d)=\cos(\beta)
\]
for some $\varphi,\psi,\alpha,\beta\in[0,\pi/2]$
with $f(-a)=\sin(\varphi), \dots f(-d)=\sin(\beta)$.
This means that
\[f_\star(x)=2^{-1/2}
\text{ for each $x\in\{\pm a, \pm b,\pm c,\pm d\}$.}\]

Now
\begin{align*}
\tfrac18\sum_{(a',b',c',d')\in\scripto} f(a')f(b')f(c')f(d')
&=
\cos(\varphi)\cos(\psi)\cos(\alpha)\cos(\beta)
\\
&\qquad
+
\sin(\varphi)\sin(\psi)\sin(\alpha)\sin(\beta)
\\
&\qquad
+
\cos(\varphi)\sin(\psi)
\cos(\alpha)\sin(\beta)
\\
&\qquad
+
\cos(\varphi)\sin(\psi)
\sin(\alpha)\cos(\beta)
\\
&\qquad
+
\sin(\varphi)\cos(\psi)
\cos(\alpha)\sin(\beta)
\\
&\qquad
+
\sin(\varphi)\cos(\psi)
\sin(\alpha)\cos(\beta)
\\
&=
\Gamma(\varphi,\psi,\alpha,\beta)
\end{align*}
where
\begin{multline*}
\Gamma(\varphi,\psi,\alpha,\beta)
=
\cos(\varphi)\cos(\psi)\cos(\alpha)\cos(\beta)
\\
+
\sin(\varphi)\sin(\psi)\sin(\alpha)\sin(\beta)
+
\sin(\varphi+\psi)\sin(\alpha+\beta).
\end{multline*}

Therefore the following lemma will complete
the proof of Proposition~\ref{prop:symmetrize}.
\end{proof}

\begin{lemma}
\begin{equation}
\max_{\varphi,\psi,\alpha,\beta\in[0,\pi/2]}\Gamma(\varphi,\psi,\alpha,\beta)
=\tfrac32.
\end{equation}
Moreover, this maximum value is attained only at
$(\tfrac\pi4,\tfrac\pi4,\tfrac\pi4,\tfrac\pi4)$.
\end{lemma}

Since \[\Gamma(\tfrac\pi4,\tfrac\pi4,\tfrac\pi4,\tfrac\pi4)
= 1+(1/\sqrt{2})^4+(1/\sqrt{2})^4 = \tfrac32,\]
the maximum value of $\Gamma$ is at least $\tfrac32$.
This point corresponds to the values taken by $f_\star$.
Compare this with $\Gamma(0,0,0,0)=1$, which represents the
extreme case when $f$ vanishes at one of each pair of antipodal points;
this ratio $(3/2)/1$ is the same $3/2$ which appears
in Corollary~\ref{cor:greater}.

\begin{proof}
We write $\Gamma $ as
\begin{align*}
\Gamma&=\cos(\phi+\psi)\cos(\alpha+\beta)+\sin(\phi+\psi)\sin(\alpha+\beta)\\
&\qquad +\cos\phi\cos\psi\sin\alpha\sin\beta+\sin\phi\sin\psi\cos\alpha\cos\beta\\
&=\cos\bigl((\phi+\psi)-(\alpha+\beta)\bigr)\\
&\qquad +\cos\phi\cos\psi\sin\alpha\sin\beta+\sin\phi\sin\psi\cos\alpha\cos\beta.
\end{align*}
Now
\begin{align*}
\cos\phi\cos\psi &=\frac {\cos(\phi+\psi)+\cos(\phi-\psi)}{2} \le \frac {1+\cos(\phi+\psi)}{2},\\
\sin\alpha\sin\beta &=\frac {-\cos(\alpha+\beta)+\cos(\alpha-\beta)}{2}
\le \frac {1-\cos(\alpha+\beta)}{2}
\end{align*}
with equality only if $\phi=\psi$ and $\alpha=\beta$,
and there are similar identities for $\sin\phi\sin\psi$ and $\cos\alpha\cos\beta$.
Therefore
\begin{align*}
\Gamma &\le \cos\bigl((\phi+\psi)-(\alpha+\beta)\bigr)
\\
&+ \tfrac 14\bigl(1+\cos(\phi+\psi)\bigr)\bigl(1-\cos(\alpha+\beta)\bigr)
+\tfrac 14\bigl(1-\cos(\phi+\psi)\bigr)\bigl(1+\cos(\alpha+\beta)\bigr)\\
&=\cos\bigl((\phi+\psi)-(\alpha+\beta)\bigr)
+\tfrac 12\bigl(1-\cos(\phi+\psi)\cos(\alpha+\beta)\bigr)\\
&=\cos\bigl((\phi+\psi)-(\alpha+\beta)\bigr)
-\tfrac12 \big(\cos\bigl((\phi+\psi)+(\alpha+\beta)\bigr)+
\cos\bigl((\phi+\psi)-(\alpha+\beta)\bigr)\big)+\tfrac {1}{2}\\
&=\tfrac12\big(\cos\bigl((\phi+\psi)-(\alpha+\beta)\bigr)-
\cos\bigl((\phi+\psi)+(\alpha+\beta)\bigr)\big)+\tfrac 12\\
&\le \tfrac 32.
\end{align*}

The value $\tfrac32$ can only be attained if
all inequalities in this derivation are equalities.
Equality in the final inequality forces $\phi+\psi+\alpha+\beta=\pi$
and $\phi+\psi=\alpha+\beta$. Together with the equalities
$\phi=\psi$ and $\alpha=\beta$ already noted, these force
$ \phi=\psi=\alpha=\beta =\frac \pi 4$.
\end{proof}

\section{Step 5: Big pieces of caps}

In this section we prove Lemma~\ref{lemma:mvv}.
While we are ultimately interested in establishing strong structural
control of near-extremal functions, here we establish a weak connection
between functions satisfying modest lower bounds
$\norm{\widehat{f\sigma}}_4\ge\delta\norm{f}_2$,
with $\delta>0$ arbitrarily small,
and characteristic functions of caps.

For each integer $k\ge 0$ choose a maximal subset $\{z_k^j\}\subset S^2$
satisfying
$|z_k^j-z_k^i|\ge 2^{-k}$ for all $i\ne j$.
Then for any $x\in S^2$ there exists $z_k^i$ such that $|x-z_k^i|\le 2^{-k}$;
otherwise $x$ could be adjoined to $\{z_k^j\}$, contradicting maximality.
Therefore the caps $\scriptc_k^j=\scriptc(z_k^j,2^{-k+1})$
cover $S^2$ for each $k$, and there exists $C<\infty$
such that for any $k$, no point of $S^2$ belongs to
more than $C$ of the caps $\scriptc_k^j$. $C$ is independent of $k$.

For $p\in [1,\infty)$, the $X_p$ norm is defined by
\begin{equation}
\norm{f}_{X_p}^4
= \sum_{k=0}^\infty
\sum_j
2^{-4k}
\big(
|\scriptc_k^j|^{-1}\int_{\scriptc_k^j}|f|^p
\big)^{4/p}.
\end{equation}
The factor $2^{-4k}$ can alternatively be written as $|\scriptc_k^j|^2$.

Define also
\begin{equation}
\Lambda_{k,j}(f)
= \big(|\scriptc_k^j|^{-1}\int_{\scriptc_k^j}|f| \big)
\big(|\scriptc_k^j|^{-1}\int_{S^2}|f|^2 \big)^{-1/2}.
\end{equation}
By H\"older's inequality,
\begin{equation}
\Lambda_{k,j}(f)\le
\big(|\scriptc_k^j|^{-1}\int_{\scriptc_k^j}|f|^2 \big)^{1/2}
\big(|\scriptc_k^j|^{-1}\int_{S^2}|f|^2 \big)^{-1/2}
= \norm{f}_{\lt(\scriptc_k^j)}/\norm{f}_{\lt(S^2)}
\le 1.\end{equation}

It is shown in Lemma~4.4 of \cite{mvv} that
$\lt\subset X_p$ for any $p<2$.
We will exploit the following refinement,
which is very closely related to a result in B\'egout and Vargas \cite{begoutvargas},
and whose somewhat tedious proof is deferred to \S\ref{section:neckpain}.
\begin{lemma} \label{lemma:neckpain}
For any $p\in[1,2)$
there exist $C<\infty$ and $\gamma>0$
such that for any $f\in\lt(S^2)$,
\begin{equation}
\norm{f}_{X_p} \le C\norm{f}_2
\sup_{k,j}  \big(\Lambda_{k,j}(f)\big)^\gamma
\end{equation}
\end{lemma}

Thus $\norm{f}_{X_p}\le C_p\norm{f}_2$ for any $f\in\lt(S^2)$.
Moreover, when the $X_p$ norm is not significantly smaller than
the $\lt$ norm, $\sup_{k,j}\Lambda_{k,j}(f)$ cannot be small.

Moyua, Vargas, and Vega \cite{mvv} have proved
\begin{proposition}
There exist $C<\infty$ and $p\in (1,2)$
such that for any $f\in\lt(S^2)$,
\begin{equation}
\norm{\widehat{f\sigma}}_{L^4(\reals^3)}
\le C\norm{f}_{X_p}.
\end{equation}
\end{proposition}
This result contains Lemma~\ref{lemma:mvv} by an elementary argument,
but we give the details for the sake of completeness.

\begin{proof}[Proof of Lemma~\ref{lemma:mvv}]
Let $\delta>0$.
Let $0\ne f\in\lt(S^2)$ and suppose that $\norm{\widehat{f\sigma}}_{L^4(\reals^3)}
\ge\delta\norm{f}_2$. For convenience, normalize so that $\norm{f}_2=1$.
The hypothesis, combined with the Proposition and the above lemma, yields
\begin{equation}
\sup_{k,j}\Lambda_{k,j}(f)\ge c\delta^{1/\gamma}.
\end{equation}
Fix $k,j$ such that $\Lambda_{k,j}(f)\ge \tfrac12 c \delta^{1/\gamma}$.
Henceforth write $\scriptc=\scriptc_k^j$.
Thus
\begin{equation}
\int_\scriptc |f|
\ge c_0\delta^{1/\gamma} |\scriptc|^{1/2}
\end{equation}
where $c_0>0$ is a constant independent of $f$.

Let $R\ge 1$.
Define $E=\{x\in\scriptc: |f(x)|\le R\}$.
Set $g = f\chi_E$ and $h = f-f\chi_E$.
Then $g,h$ have disjoint supports, $g+h=f$,
$g$ is supported on $\scriptc$,
and $\norm{g}_\infty\le R$.
Now
$|h(x)|\ge R$ for almost every $x\in\scriptc$
for which $h(x)\ne 0$, so
\begin{equation}
\int_\scriptc |h|
\le R^{-1}\int_\scriptc|h|^2
\le R^{-1}\norm{f}_2^2
= R^{-1}.
\end{equation}
Define $R$ by
$R^{-1}=
\tfrac12 c_0\delta^{1/\gamma} |\scriptc|^{1/2}$.
Then
\begin{equation}
\int_\scriptc |g| = \int_\scriptc|f|-\int_\scriptc|h|
\ge
\tfrac12 c_0\delta^{1/\gamma} |\scriptc|^{1/2}.
\end{equation}
By H\"older's inequality, since $g$ is supported on $\scriptc$,
\begin{equation}
\norm{g}_2 \ge |\scriptc|^{-1+\frac12}\norm{g}_{L^1(\scriptc)}
\ge c\delta^{1/\gamma}
= c\delta^{1/\gamma}\norm{f}_2.
\end{equation}
Thus the decomposition $f=g+h$ satisfies the conclusions of
Lemma~\ref{lemma:mvv}, with $\eta_\delta$ proportional to $\delta^{1/\gamma}$,
and $C_\delta$ proportional to $\delta^{-1/\gamma}$.
\end{proof}

\section{Analytic preliminaries}

\subsection{On near-extremals}
\begin{lemma} \label{lemma:nearextremal}
Let $f=g+h\in\lt(S^2)$.
Suppose that $g\perp h$, $g\ne 0$,
and that $f$ is $\delta$--nearly extremal
for some $\delta\in (0,\tfrac14]$.
Then
\begin{equation} \label{nearextremalconclusion}
\frac{ \norm{h}_2 }{ \norm{f}_2 }
\le
C\max\Big(
\frac{ \norm{h\sigma*h\sigma}_2^{1/2} }{ \norm{h}_2 }\ ,
\
\delta^{1/2}
\Big).
\end{equation}
Here $C<\infty$ is a constant independent of $g,h$.
\end{lemma}

\begin{proof}
The inequality is invariant under multiplication of $f$
by a positive constant, so we may assume without loss of generality that $\norm{g}_2=1$.
We may assume that $\norm{h}_2>0$, since otherwise the conclusion is trivial.
Define $y=\norm{h}_2$
and
\begin{equation}\eta = \norm{h\sigma*h\sigma}_2^{1/2}/\Sbest\norm{h}_2.\end{equation}
If $\eta>\tfrac12$
then \eqref{nearextremalconclusion} holds trivially
with $C=2/\Sbest$,
for the left-hand side cannot exceed $1$ since $f=g+h$ with $g\perp h$.

Since $\norm{f\sigma*f\sigma}_2^{1/2}$ is a constant multiple of
$\norm{\widehat{f\sigma}}_4$,
the functional $f\mapsto \norm{f\sigma*f\sigma}_2^{1/2}$ satisfies the triangle inequality.
Therefore
\begin{equation}
(1-\delta)^4\Sbest^4\norm{f}_2^4
\le
\norm{f\sigma*f\sigma}_2^2
\le \big(
\norm{g\sigma*g\sigma}_2^{1/2}
+\norm{h\sigma*h\sigma}_2^{1/2}
\big)^4
\le \Sbest^4
(1+\eta y)^4.
\end{equation}
Since $g\perp h$,
$\norm{f}_2^2=1+y^2$ and therefore
\begin{equation}
(1-\delta)(1+y^2)^{1/2}
\le 1+\eta y.
\end{equation}
Squaring gives
\begin{equation}
(1-2\delta)(1+y^2)\le 1+2\eta y + \eta^2y^2.
\end{equation}
Since $\delta\in(0,\tfrac14]$ and $\eta\le\tfrac12$,
\begin{equation}
\tfrac12 y^2\le 2\delta + 2\eta y + \eta^2y^2
\le 2\delta + 2\eta y + \tfrac14 y^2
\end{equation}
whence either $y^2\le 16\delta$ or $y\le 16\eta$.

Substituting the definitions of $y,\eta$,
and majorizing $\norm{h}_2/\norm{f}_2$
by $\norm{h}_2/\norm{g}_2$, yields the stated conclusion.
\end{proof}

\subsection{Simple bilinear convolution estimates}


\begin{lemma} \label{lemma:outerbrute}
Let $f\in\lt(S^2)$ be nonnegative,
and satisfy $\norm{f}_2\le 1$.
Let $z\in S^2$ and $\eps>0$.
Let $R\ge 1$ and $0<\rho\le 1$.
Then for any $R\in\reals^+$,
\begin{equation}
\norm{f\sigma*f\sigma}_{\lt(\{|x|>2-\eps\})}
\le CR^2\eps^{1/2}\rho
+C\big(\int_{f(x)\ge R} f^2(x)\,dx\big)^{1/2}
+ C\big(\int_{|x-z|\ge \rho} f^2(x)\,dx \big)^{1/2}.
\end{equation}
\end{lemma}

\begin{proof}
Decompose $f=g+h$ where
$g,h$ are nonnegative,
\[
\norm{h}_2
\le
\big(\int_{f(x)\ge R} f^2(x)\,dx\big)^{1/2}
+ \big(\int_{|x-z|\ge \rho} f^2(x)\,dx \big)^{1/2}
\]
and $\norm{g}_2\le 1$,
$\norm{g}_\infty \le R$,
and $g$ is supported on $\{x\in S^2: |x-z|\le \rho\}$.
Then
\[
g\sigma*g\sigma(x)\le R^2 \sigma*\sigma(x)
\le CR^2|x|^{-1}
\]
for $|x|<2$, and $=0$ otherwise.
Moreover, $g\sigma* g\sigma$ is supported
in
$\{x: |x-2z|<2\rho \}$.
The $\lt(\reals^3)$ norm of $|x|^{-1}$ over the intersection
of this region with $\{x:|x|>2-\eps\}$ is $\le C\rho\eps^{1/2}$.
This gives the bound $CR^2\rho\eps^{1/2}$ for
$\norm{g\sigma*g\sigma}_2$. Since $\norm{g}_2\le 1$,
the general inequality
\[
\norm{F\sigma*G\sigma}_{\lt(\reals^3)}
\le C\norm{F}_2\norm{G}_2
\]
gives the required bound for both $g\sigma*h\sigma$
and $h\sigma* h\sigma$.
\end{proof}

\begin{corollary} \label{cor:neglectlargerx}
Let $\{f_\nu\}$ be a sequence of real-valued functions which
are upper even-normalized above with respect to a sequence
of caps $\scriptc_\nu$ of radii $r_\nu$.
If
\begin{equation}\delta_\nu/r_\nu^2\to 0,\end{equation}
then
\begin{equation}
\int_{|x|>2-\delta_\nu} (|f_\nu|\sigma*|f_\nu|\sigma)^2\,dx
\to 0 \text { as } \nu\to\infty.
\end{equation}
\end{corollary}

\begin{lemma} \label{lemma:innerbrute}
Let $f\in\lt(S^2)$ be a function which is upper even-normalized with respect
to a cap $\scriptc$ of radius $r$. Then
for all $R\ge 1$,
\begin{equation}
\int_{R^{1/2}r\le |x| \le 2-Rr^2} |(f\sigma*f\sigma)(x)|^2\,dx
\le \Psi(R)
\end{equation}
where $\Psi(R)\to 0$ as $R\to\infty$,
and $\Psi$ depends only on the function $\Theta$ in the normalization
inequalities \eqref{eq:normalization1},\eqref{eq:normalization2},
not on $r$.
\end{lemma}

\begin{proof}
It suffices to prove this for $r$ small, $R$ large,
and $Rr^2$ uniformly bounded.
Let $\scriptc=\scriptc(z,r)$ have center $z\in S^2$.
Let $A\in[1,\infty)$ and decompose
$f = g_+ + h_+ +g_- + h_-$
where $g_+,g_-$ are supported respectively in
$\scriptc(z,Ar)$ and $\scriptc(-z,Ar)$,
$\norm{h_+}_2\le\Theta(A)$ and
$\norm{h_-}_2\le\Theta(A)$, where
$\Theta(A)\to 0$ as $A\to\infty$.

Expand $f\sigma*f\sigma$ as a sum of the resulting $16$ terms.
The terms
$g_+\sigma*g_+\sigma$ and $g_-\sigma*g_-\sigma$
are supported where $|x|>2-CA^2r^2$.
If we choose $A$ so that $CA^2<R$
then these vanish identically in the region
$|x|\le 2-Rr^2$.
The (two) terms $g_+\sigma*g_-\sigma$ are supported where $|x|\le CAr$.
Therefore they also contribute nothing, provided that $CAr\le R^{1/2}r$.

Each of the remaining terms involves at least one factor of $h_+$ or of $h_-$.
Since $\norm{F\sigma*G\sigma}_{\lt(\reals^3)}\le C\norm{F}_2\norm{G}_2$
for all $F,G\in\lt(S^2)$,
and since $g_\pm,h_\pm=O(1)$ in $\lt(S^2)$ norm,
each of these terms is $O(\norm{h_\pm}_2)$.
Therefore
\begin{equation}
\int_{R^{1/2}r\le |x| \le 2-Rr^2} |f\sigma*f\sigma(x)|^2\,dx
\le C\Theta(A)^2
\end{equation}
for any $A$ which satisfies $CA^2< R$.
This completes the proof,
provided that $Rr^2=O(1)$.
\end{proof}

The set of all caps can be made into a metric space.
Define the distance $\rho$ from $\scriptc(y,r)$ to $\scriptc(y',r')$
to be the Euclidean distance from $(y/r,\log(1/r))$ to $(y'/r',\log(1/r'))$
in $\reals^3\times\reals^+$.
Note that for instance when $r=r'$, the distance is $r^{-1}|y-y'|$,
so this distance has the natural scaling.
If $y=y'$, then the distance is $|\log(r/r')|$; this has the natural property
that it depends only on the {\em ratio} of the two radii.
The definition ensures that this is truly a metric.

For any metric space $(X,\rho)$ and any equivalence relation
$\equiv$ on $X$, the function $\varrho([x],[y])
=\inf_{x'\in[x],y'\in[y]}\rho(x',y')$
is a metric on the set of equivalence classes $X/\equiv$.
Let $\scriptm$ be the set of all caps $\scriptc\subset S^2$
modulo the equivalence relation $\scriptc\equiv -\scriptc$,
where $-\scriptc=\{-z: z\in\scriptc\}$.
Then the following defines a metric on $\scriptm$.
\begin{definition} \label{defn:capsmetricspace}
For any two caps $\scriptc,\scriptc'\subset S^2$,
\begin{equation}
\varrho([\scriptc],[\scriptc'])
= \min(\rho(\scriptc,\scriptc'),\rho(-\scriptc,\scriptc'))
\end{equation}
where $[\scriptc]$ denotes the equivalence class
$[\scriptc]=\{\scriptc,-\scriptc\}\in\scriptm$.
\end{definition}
We will also write $\varrho(\scriptc,\scriptc')=\varrho([\scriptc],[\scriptc'])$.

\begin{lemma} \label{lemma:distantcaps} \label{lemma:twocaps}
For any $\eps>0$ there exists $\rho<\infty$ such that
\begin{equation}
\norm{\chi_{\scriptc}\sigma*\chi_{\scriptc'}\sigma}_{L^2(\reals^3)}
<\eps|\scriptc|^{1/2}|\scriptc'|^{1/2}
\end{equation}
whenever
\begin{equation*}\varrho(\scriptc,\scriptc')>\rho.\end{equation*}
\end{lemma}

\begin{proof}
Let $\scriptc=\scriptc(z,r)$, $\tilde\scriptc=\scriptc(\tilde z,\tilde r)$.
Set $f=|\scriptc|^{-1/2}\chi_{\scriptc}\le Cr^{-1}\chi_{\scriptc}$,
$\tilde f=|\tilde\scriptc|^{-1/2}\chi_{\tilde\scriptc}\le C\tilde r^{-1}\chi_{\tilde\scriptc}$.
Without loss of generality, $\tilde r\le r$.
We may suppose that $\tilde r\ll 1$; otherwise
the caps are not far apart.
We will also assume at first that no points are nearly antipodal,
that is, that $|x+\tilde x|\ge \delta$
for all $x\in\scriptc$ and $\tilde x\in\tilde\scriptc$,
for some fixed constant $\delta>0$;  we will return to this point later.

Consider first the case where $r\sim\tilde r$.
Then we may assume that $|z-\tilde z|\ge 10r$, say.
Then $f\sigma*\tilde f\sigma$
has $L^\infty$ norm $\le Cr^{-2} \cdot r/|z-\tilde z|$,
and is supported in a three-dimensional cylinder
whose base has radius $Cr$
and whose height is $\le Cr^2 + Cr|z-\tilde z|\le Cr|z-\tilde z|$.
The volume of this cylinder is $\le Cr^3 |z-\tilde z|$.
In all,
\begin{equation}
\norm{f\sigma*\tilde f\sigma}_{\lt(\reals^3)}
\le Cr^{-1}|z-\tilde z|^{-1} \cdot r^{3/2}|z-\tilde z|^{1/2}
= C\big( r/|z-\tilde z|\big)^{1/2},
\end{equation}
which is small precisely when the caps are far apart.

Consider next the case where $\tilde r\ll r$,
and still $|z-\tilde z|\ge 10r$.
Then the $L^\infty$ norm is $\le Cr^{-1}\tilde r^{-1} \cdot \tilde r|z-\tilde z|^{-1}$.
The support is contained in a tubular neighborhood of a (translated) cap of
radius $Cr$; this tubular neighborhood has width $\le C\tilde r |z-\tilde z|$.
Hence the volume of the support is
$\le Cr^2\tilde r|z-\tilde z|$.
Consequently
\begin{multline}
\norm{f\sigma*\tilde f\sigma}_{\lt(\reals^3)}
\le Cr^{-1}\tilde r^{-1}|z-\tilde z|^{-1}
\cdot
r\tilde r^{1/2}|z-\tilde z|^{1/2}
\\
=
C\big(\tilde r/|z-\tilde z|\big)^{1/2}
\le C(\tilde r/r)^{1/2}.
\end{multline}

Consider next the case where $\tilde r\ll r$,
and $|z-\tilde z|\le 10r$.
It suffices to replace $f$ by its restriction $F$ to
the complement of the cap $\tilde\scriptc^\star$
centered at $\tilde z$ of radius $10r^{3/4}\tilde r^{1/4}$;
for
\begin{equation}
\norm{f-F}_2\le Cr^{-1}r^{3/4}\tilde r^{1/4} = C(\tilde r/r)^{1/4}\ll 1.
\end{equation}
$F\sigma*\tilde f\sigma$ is supported in a region of volume
$\le Cr^3\tilde r$,
and as is easily verified,
\begin{equation}
\norm{F\sigma*\tilde f\sigma}_\infty
\le Cr^{-1}\tilde r^{-1}\cdot \big(\tilde r/r^{3/4}\tilde r^{1/4} \big)
= Cr^{-7/4}\tilde r^{-1/4}.
\end{equation}
Therefore
\begin{equation}
\norm{ F\sigma*\tilde f\sigma }_2
\le
Cr^{-7/4}\tilde r^{-1/4}
\cdot
(r^3\tilde r)^{1/2}
=Cr^{-1/4}\tilde r^{1/4}\ll 1.
\end{equation}

It only remains to handle caps which are nearly antipodal. But this follows
from the non-antipodal case by the identity
\begin{equation}
\norm{f\sigma\,*\,g\sigma}_2 = \norm{f^\star\sigma\,*\,g\sigma}_2
\end{equation}
where $f^\star(x)\equiv \overline{f(-x)}$.
\end{proof}

\subsection{Fourier integral operators}

Here we discuss another ingredient required for the proof of Lemma~\ref{lemma:smallcrossterm},
certain estimates which rely on cancellation, in contrast to those in the preceding section.

For $0<\rho\lesssim 1$
define $T_\rho:\lt(S^2)\to \lt(S^2)$ by
\begin{equation}
T_\rho f(x) = \int f(y)\,d\mu_{x,\rho}(y)
\end{equation}
where $\mu_{x,\rho}$ is arc length measure on the circle
$\{y\in S^2: |y-x|=\rho\}$,
normalized to be a probability measure.

Let $\Delta$ denote the spherical Laplacian.
\begin{lemma}
\begin{equation} \label{FIObound}
\norm{T_\rho f}_{\lt(S^2)}
\le C\norm{(I-\rho^2\Delta)^{-1/4}f}_{\lt(S^2)}
\end{equation}
uniformly for all $\rho>0$ and all $f\in\lt(S^2)$.
\end{lemma}

\begin{proof}[Sketch of proof]
There are three elements in the proof of \eqref{FIObound}.

(i) Consider any fixed $\rho\in(0,2)$.
Define $\Phi_\rho(x,y) = |x-y|^2-\rho^2$.
Then the $3\times 3$ matrix
\begin{equation} \label{FIOmatrixcondition}
\begin{pmatrix}
0 & \partial\Phi_\rho/\partial x \\
\partial\Phi_\rho/\partial y & \partial^2\Phi_\rho/\partial x\partial y
\end{pmatrix}
\end{equation}
is nonsingular for any $(x,y)$ satisfying $\Phi_\rho(x,y)=0$.
This is a straightforward computation,
easily done by taking advantage of rotational symmetry
to reduce to a computation of Taylor expansions
about $x=(0,0,1)$ and $y=(\cos(\theta),0,\sin(\theta))$.

(ii)
$T_\rho$ is defined by integration against
a smooth density on $\{(x,y)\in S^2\times S^2:
\Phi_\rho(x,y)=0\}$.
As discussed on pages 188-9 of \cite{sogge},
the nonsingularity of the matrix \eqref{FIOmatrixcondition}
implies that
$T_\rho$ is a Fourier integral operator of
order $-(n-1)/2=-1/2$ on $S^n=S^2$.
Any such operator is smoothing of order $1/2$
in the scale of $\lt$ Sobolev spaces \cite{sogge}.

(iii) If $T_\rho$ is rewritten with appropriate
normalizations in coordinates adapted to
any cap $\scriptc(z,\rho)$, then the inequality holds
uniformly in $\rho$.
The only issue here is as $\rho\to 0$, but plainly in that
situation there is a limiting operator on $\reals^2$,
$f\mapsto \int_{S^1}f(x-y)\,d\mu(y)$ where
$\mu$ is arc length measure on $S^1\subset\reals^2$.
This limiting operator is again a Fourier integral
operator of order $-1/2$. It follows that the bounds
are uniform after rescaling.
Reversal of the rescaling introduces the factor
$\rho^2$ to $\Delta$ in the inequality.
\end{proof}

The operators $T_\rho$ are related to our bilinear convolutions:
For $f\in\lt(S^2)$ and $x\in\reals^3$ satisfying $0<|x|<2$,
\begin{equation}
(f\sigma*\sigma)(x) = c|x|^{-1} T_\rho f(x/|x|)
\end{equation}
where
\begin{equation}
\rho^2 + |x/2|^2=1.
\end{equation}
Define $e_\xi(x) = e^{ix\cdot\xi}$, for $x\in \reals^3$  and $\xi\in\complex$
(and in particular for $x\in S^2$). There is the more general identity
\begin{equation} \label{switchcharacter}
(f\sigma*e_{i\xi}\sigma)(x) =
e_{i\xi}(x)
\Big(e_{-i\xi}f\sigma \ * \ \sigma \Big)(x)
=
c|x|^{-1}
e_{i\xi}(x)
T_\rho(e_{-i\xi}f)(x).
\end{equation}

Suppose that $g\in\lt(S^2)$ takes the form
$g(x) = \int_H a(\xi)e_{i\xi}(x)\,d\nu(\xi)$
where $H\subset\reals^3$ is a two-dimensional subspace,
$\nu$ is Lebesgue measure on $H$,
and $a\in\lt(H)$.
Then
\begin{equation}
(f\sigma*g\sigma)(x)
= c|x|^{-1} \int_H a(\xi) e_{i\xi}(x)
T_\rho(e_{-i\xi}f)(x)\,d\nu(\xi).
\end{equation}
For $t\in(0,2)$ define $\rho(t)>0$ by
\begin{equation}\rho(t)^2+ (t/2)^2=1.\end{equation}
Then for any interval $I\subset(0,2)$,
\begin{equation} \label{Iversion}
\begin{split}
\int_{|x|\in I} |(f\sigma*g\sigma)(x)|^2\,dx
&\le C
\int_{I}
t^{-2}
\Big\| \int_{H} |a(\zeta)| \cdot|T_{\rho(t)}(e_{-i\zeta}f)| \,d\zeta
\Big\|_{\lt(S^2)}^2
t^2\,dt
\\
&=C
\int_{I}
\Big\| \int_{H} |a(\zeta)| \cdot|T_{\rho(t)}(e_{-i\zeta}f)| \,d\zeta
\Big\|_{\lt(S^2)}^2
\,dt.
\end{split}
\end{equation}

\subsection{Fourier coefficient estimates in terms of the spherical Laplacian}

The following routine lemma is convenient because it provides an
intrinsic characterization of expressions which arise in the analysis.
The proof relies on the machinery of pseudodifferential operators,
and is left to the reader.

\begin{lemma} \label{lemma:DeltaversusFT}
Let $\scriptc$ be a cap of radius $\varrho\le \tfrac12$.
Let $\phi$ be the rescaling map associated with $\scriptc$.
Let $f$ be supported in $\scriptc\cup (-\scriptc)$.
Then for any $t\in\reals$ and $0<r\le \varrho$,
\begin{equation}
C^{-1}
\norm{(I-r^2\Delta)^{t/2}f}_{\lt(S^2)}^2
\le
\int_{\reals^2} |\widehat{\phi^* f}(\xi)|^2
(1+|r\varrho^{-1}\xi|^2)^{t}
\,d\xi
\le C
\norm{(I-r^2\Delta)^{t/2}f}_{\lt(S^2)}^2.
\end{equation}
Here $C\in(0,\infty)$ depends on $t$ but not on $f,r,\varrho,\scriptc$.
\end{lemma}

\section{Step 6A: A decomposition algorithm}
The following iterative procedure may be applied
to any nonnegative function $f\in\lt(S^2)$ of positive norm.

\medskip
\noindent{\bf Decomposition algorithm.}
Initialize by setting $G_0=f$, and $\eps_0=1/2$.

Step $\nu$: The inputs for step $\nu$ are a nonnegative function
$G_\nu\in\lt(S^2)$ and a positive number $\eps_\nu$.
Its outputs are functions $f_\nu,G_{\nu+1}$
and nonnegative numbers $\eps_\nu^\star,\eps_{\nu+1}$.
If $\norm{G_\nu\sigma*G_\nu\sigma}_2=0$ then $G_\nu=0$ almost everywhere. The
algorithm then terminates, and we
define  $\eps_\nu^\star=0$, $f_\nu=0$,
and $G_\mu=f_\mu=0$, $\eps_\mu=0$ for all $\mu>\nu$.

If $0<\norm{G_\nu\sigma*G_\nu\sigma}_2<\eps_\nu^2 \Sbest^2\norm{f}_2^2$
then replace $\eps_\nu$ by $\eps_\nu/2$; repeat until the first time that
$\norm{G_\nu\sigma*G_\nu\sigma}_2\ge \eps_\nu^2 \Sbest^2\norm{f}_2^2$.
Define $\eps_\nu^\star$ to be this value of $\eps_\nu$.
Then
\begin{equation}
(\eps_\nu^\star)^2 \Sbest^2\norm{f}_2^2
\le\norm{G_\nu\sigma*G_\nu\sigma}_2
\le 4(\eps_\nu^\star)^2 \Sbest^2\norm{f}_2^2.
\end{equation}

Apply Lemma~\ref{lemma:mvv} to obtain a cap $\scriptc_\nu$
and a decomposition $G_\nu=f_\nu + G_{\nu+1}$
with disjointly supported nonnegative summands satisfying
$f_\nu \le C_{\nu} \norm{f}_2
|\scriptc_\nu|^{-1/2}\chi_{\scriptc_\nu}$,
and
$\norm{f_\nu}_2\ge\eta_{\nu}\norm{f}_2$.
Here $C_\nu,\eta_\nu$ are bounded above and below, respectively,
by quantities which depend only on
$\norm{G_\nu\sigma*G_\nu\sigma}_2^{1/2}/\norm{G_\nu}_2
\ge\eps_\nu^\star$.
Define $\eps_{\nu+1}=\eps_\nu^\star$, and move on to step $\nu+1$.
\qed

\medskip
It is important for our application to observe
that if $f$ is even then at every step, $f_\nu$ may likewise be chosen to be even.
The upper bound for $f_\nu$ then becomes
\begin{equation}
f_\nu\le C_\nu |\scriptc|^{-1/2}
\chi_{\scriptc\cup-\scriptc}.
\end{equation}
Henceforth the algorithm will be applied only to even functions, and
we will always choose all $f_\nu$ to be even.

\medskip

If the algorithm terminates at some finite step $\nu$, then
a finite decomposition $f = \sum_{k=0}^\nu f_k$ results.

\begin{lemma}
Let $f\in\lt(S^2)$ be a nonnegative function with positive norm.
If the decomposition algorithm never terminates for $f$,
then $\eps_\nu^\star\to 0$ as $\nu\to\infty$,
and
$\sum_{\nu=0}^N f_\nu\to f$ in $\lt$ as $N\to\infty$.
\end{lemma}

\begin{proof}
Assume without loss of generality that $\norm{f}_2=1$.
The functions $f_\nu$ have disjoint supports and hence are pairwise orthogonal,
and $\sum_\nu f_\nu\le f$, so
$\sum_\nu\norm{f_\nu}_2^2\le\norm{f}_2^2$.
Since the sequence $\eps_\nu^\star$ is nonincreasing
and $\norm{f_\nu}_2/\norm{f}_2$ is bounded below by a function of $\eps_\nu^\star$,
this forces $\eps_\nu^\star\to 0$.

The second conclusion is equivalent to $\norm{G_N}_2\to 0$.
$\norm{f_\nu}_2$ is bounded below, according to Lemma~\ref{lemma:mvv},
by a function of  $\norm{G_\nu\sigma*G_\nu\sigma}_2$.
Since $\sum_\nu\norm{f_\nu}_2^2<\infty$,
$\norm{f_\nu}_2\to 0$ and therefore
$\norm{G_\nu\sigma*G_\nu\sigma}_2\to 0$.
By construction, $G_{\nu+1}(x)\le G_\nu(x)$ for every $x\in S^2$,
so $G(x)=\lim_{\nu\to\infty} G_\nu(x)$ exists and
$\norm{G\sigma*G\sigma}_2\le \norm{G_\nu\sigma*G_\nu\sigma}_2$
for all $\nu$. Thus $G\sigma*G\sigma\equiv 0$, so $G\equiv 0$.
This forces $\norm{G_\nu}_2\to 0$, by the dominated convergence
theorem.
\end{proof}

\medskip
For general $f$, this decomposition may be highly inefficient.
But if $f$ is nearly extremal for the inequality \eqref{secondSversion}
then more useful properties hold.
\begin{lemma} \label{lemma:goodfnubound}
There exists a continuous function $\theta:(0,1]\to(0,\infty)$
such that
for any $\eps>0$ there exists $\delta>0$
such that for any $\delta$-nearly extremal nonnegative function $f\in\lt(S^2)$
satisfying $\norm{f}_2=1$,
the functions $f_\nu,G_\nu$ associated by the decomposition algorithm
to $f$ satisfy
\begin{equation}
\norm{f_\nu}_2\ge \theta(\norm{G_\nu}_2)
\text{ for any index $\nu$ such that } \norm{G_\nu}_2\ge\eps.
\end{equation}
\end{lemma}
This is a direct consequence of Lemmas~\ref{lemma:mvv} and
\ref{lemma:nearextremal}.
It is essential for applications below that $\theta$ be independent of $\eps$.

If $f$ is nearly extremal, then the norms of $f_\nu,G_\nu$
enjoy upper bounds independent of $f$, for all except very large $\nu$.
\begin{lemma} \label{lemma:decompuniformdecay}
There exist a sequence of positive constants $\gamma_\nu\to 0$
and a function $N:(0,\tfrac12]\to\integers^+$ satisfying
$N(\delta)\to\infty$ as $\delta\to 0$
such that for any nonnegative $f\in\lt(S^2)$,
if $f$ is $\delta$--nearly extremal then
the quantities $\eps_\nu^\star$ obtained when the decomposition algorithm is applied to
$f$ satisfy
\begin{align}
\eps_\nu^\star\le\gamma_\nu  &\text{ for all } \nu\le N(\delta).
\\
\norm{G_\nu}_2\le\gamma_\nu\norm{f}_2 &\text{ for all } \nu\le N(\delta).
\\
\norm{f_\nu}_2\le\gamma_\nu\norm{f}_2 &\text{ for all } \nu\le N(\delta).
\end{align}
\end{lemma}
This holds whether or not the algorithm terminates for $f$.

\begin{proof}
\begin{equation}
\Sbest^2\norm{G_\nu}_2^2
\ge\norm{G_\nu\sigma*G_\nu\sigma}_2
\ge \eps_\nu^\star{}^2\Sbest^2\norm{f}_2^2
=\Big(\eps_\nu^\star{}^2\norm{f}_2^2/\norm{G_\nu}_2^2\Big)\Sbest^2\norm{G_\nu}_2^2,
\end{equation}
so $\eps_\nu^\star\le \norm{G_\nu}_2/\norm{f}_2$.
Thus the second conclusion implies the first.
Since $\norm{f_{\nu}}_2\le\norm{G_\nu}_2$, it also implies the third.

We recall two facts. Firstly,
Lemma~\ref{lemma:nearextremal}, applied to $h=G_\nu$ and $g=f_0+\cdots+f_{\nu-1}$,
asserts that there are constants $c_0,C_1\in\reals^+$ such that whenever
$f\in\lt$ is $\delta$-nearly extremal,
either
$\norm{G_\nu\sigma*G_\nu\sigma}_2\ge c_0\norm{G_\nu}_2^4\norm{f}_2^{-2}$,
or
$\norm{G_\nu}_2\le C_1\delta^{1/2}\norm{f}_2$.
Secondly,
according to Lemma~\ref{lemma:mvv},
there exists a nondecreasing function $\rho:(0,\infty)\to(0,\infty)$
satisfying $\rho(t)\to 0$ as $t\to 0$
such that for every nonzero $f\in\lt$ and any $\nu$,
if $\norm{G_\nu\sigma*G_\nu\sigma}_2\ge t\norm{G_\nu}_2^2$
then $\norm{f_\nu}_2^2\ge\rho(t)\norm{G_\nu}_2^2$.

Choose a sequence $\{\gamma_\nu\}$ of positive numbers which tends
monotonically to zero, but does so sufficiently slowly to satisfy
\begin{equation}\nu\gamma_\nu^2\rho(c_0\gamma_\nu^2)>1 \text{ for all $\nu$.}\end{equation}
Define $N(\delta)$ to be the largest integer satisfying
\begin{equation}\gamma_{N(\delta)}\ge C_1\delta^{1/2}.\end{equation}
$N(\delta)\to\infty$ as $\delta\to 0$ because $\gamma_\nu>0$ for all $\nu$.

Let $f,\delta$ be given.
Suppose that $\nu\le N(\delta)$.
We argue by contradiction, supposing that
$\norm{G_\nu}_2>\gamma_\nu\norm{f}_2$.
Then by definition of $N(\delta)$,
$\norm{G_\nu}_2>C_1\delta^{1/2}\norm{f}_2$.
By the above dichotomy,
\begin{equation}
\norm{G_\nu\sigma*G_\nu\sigma}_2\ge c_0\norm{G_\nu}_2^4\norm{f}_2^{-2}
\ge c_0\gamma_\nu^2\norm{G_\nu}_2^2.
\end{equation}

By the second fact reviewed above,
\begin{equation}
\norm{f_\nu}_2^2
\ge\rho(c_0\gamma_\nu^2)\norm{G_\nu}_2^2
\ge\gamma_\nu^2\rho(c_0\gamma_\nu^2)\norm{f}_2^2.
\end{equation}
Since $\norm{G_\mu}_2\ge\norm{G_\nu}_2$
for all $\mu\le\nu$,
the same lower bound follows for $\norm{f_\nu}_2^2$
for all $\mu\le\nu$.
Since the functions $f_\mu$ are pairwise orthogonal,
$\sum_{\mu\le\nu}\norm{f_\mu}_2^2\le\norm{f}_2^2$,
and consequently
$\nu\gamma_\nu^2\rho(c_0\gamma_\nu^2)\le 1$,
a contradiction.
\end{proof}

The next lemma also follows directly from the decomposition algorithm coupled with
Lemma~\ref{lemma:mvv}.
\begin{lemma}
\label{lemma:Galsosmall}
For any $\eps>0$ there exist $\delta_\eps>0$ and $C_\eps<\infty$ such that
for every $\delta_\eps$--nearly extremal nonnegative function $f\in\lt$,
the functions $f_\nu,G_\nu$ associated to $f$ by the decomposition algorithm
satisfy
\newline
(i)
For any $\nu$,
if $\norm{G_\nu}_2\ge\eps\norm{f}_2$ then
there exists a cap $\scriptc_\nu\subset S^2$ such that
\begin{equation}
f_\nu\le C_\eps\norm{f}_2|\scriptc_\nu|^{-1/2}\chi_{\scriptc_\nu\cup-\scriptc_\nu}.
\end{equation}
\newline
(ii)
If
$\norm{G_\nu}_2\ge\eps\norm{f}_2$
then
$\norm{f_\nu}_2\ge\delta_\eps\norm{f}_2$.
\end{lemma}

\section{Step 6B: A geometric property of the decomposition}

We have established inequalities concerning the $\lt$ norms
of the functions $f_\nu,G_\nu$ which the decomposition algorithm yields,
based on quite general principles and a single analytic fact,
Lemma~\ref{lemma:mvv}, concerning the particular inequality which we are studying.
We next establish an additional inequality of a geometric nature,
based on a single additional fact,
the weak interaction of distant caps in the sense of Lemma~\ref{lemma:twocaps}.

\begin{lemma} \label{lemma:metricspace}
In any metric space, for any $N,r$, any finite set
$S$ of cardinality $N$ and diameter equal to $r$
may be partitioned into two disjoint nonempty subsets $S=S'\cup S''$
such that
$\distance(S',S'')\ge r/2N$.
Moreover, given two points $s',s''\in S$ satisfying
$\distance(s',s'')=r$, this partition can be constructed
so that $s'\in S'$ and $s''\in S''$.
\end{lemma}

\begin{proof}
Consider the metric balls $B_k$ centered at $s'$
of radii $kr/2N$ for $k=1,2,\cdots, 2N$.
By the pigeonhole principle, there exists $k$ such that
$(B_{k+1}\setminus B_k)\cap S=\emptyset$. Set $S'=B_k\cap S$,
$S'' = S\setminus S'$. The triangle inequality
yields the conclusion.
\end{proof}

\begin{lemma} \label{lemma:controlgeometry}
For any $\eps>0$ there exist $\delta>0$ and $\lambda<\infty$
such that for any $0\le f\in\lt(S^2)$
which is  $\delta$--nearly extremal,
the summands $f_\nu$ produced by the decomposition algorithm
and the associated caps $\scriptc_\nu$ satisfy
\begin{equation}
\varrho(\scriptc_j,\scriptc_k)\le \lambda
\text{ whenever
$\norm{f_j}_2\ge\eps\norm{f}_2$ and $\norm{f_k}_2\ge\eps\norm{f}_2$. }
\end{equation}
\end{lemma}
Here $\varrho$ is the distance between $\scriptc_j\cup-\scriptc_j$
and $\scriptc_k\cup-\scriptc_k$, as defined in Definition~\ref{defn:capsmetricspace}.

\begin{proof}
It suffices to prove this for all sufficiently small $\eps$.
Let $f$ be a nonnegative $\lt$ function which satisfies
$\norm{f}_2=1$ and
is $\delta$--nearly extremal
for a sufficiently small $\delta=\delta(\eps)$,
and let $\{G_\nu,f_\nu\}$ be associated to $f$ via the decomposition algorithm.
Set $F = \sum_{\nu=0}^N f_\nu$.

Suppose that  $\norm{f_{j_0}}_2\ge\eps$ and $\norm{f_{k_0}}_2\ge\eps$.
Let $N$ be the smallest integer such that $\norm{G_{N+1}}_2<\eps^3$.
Since $\norm{G_\nu}_2$ is a nonincreasing function of $\nu$,
and since $\norm{f_\nu}_2\le\norm{G_\nu}_2$,
necessarily $j_0,k_0\le N$.
Moreover, by Lemma~\ref{lemma:decompuniformdecay},
there exists $M_\eps<\infty$ depending only on $\eps$ such that
$N\le M_\eps$.
By Lemma~\ref{lemma:Galsosmall}, if $\delta$ is chosen to be a sufficiently small
function of $\eps$ then since $\norm{G_\nu}_2\ge\eps^3$ for all $\nu\le N$,
$f_\nu\le \theta(\eps) |\scriptc|^{-1/2}\chi_{\scriptc\cup-\scriptc}$
for all such $\nu$, where $\theta$ is a continuous, strictly positive
function on $(0,1]$.

Now let $\lambda<\infty$ be a large quantity to be specified.
It suffices to show that if $\delta(\eps)$ is sufficiently
small,
an assumption that $\varrho(\scriptc_j,\scriptc_k)>\lambda$
implies an upper bound,  which depends only on $\eps$, for $\lambda$.

Lemma~\ref{lemma:metricspace} yields a decomposition
$F = F_1+F_2=\sum_{\nu\in S_1}f_\nu + \sum_{\nu\in S_2}f_\nu$
where $[0,N]=S_1\cup S_2$ is a partition of $[0,N]$,
$j_0\in S_1$, $k_0\in  S_2$,
and
$\varrho(\scriptc_j,\scriptc_k) \ge \lambda/2N\ge \lambda/2M_\eps$
for all $j\in S_1$ and $k\in S_2$.
Certainly
$\norm{F_1}_2\ge\norm{f_{j_0}}_2\ge \eps$ and similarly $\norm{F_2}_2\ge\eps$.
The convolution cross term satisfies
\begin{equation}
\norm{F_1\sigma*F_2\sigma}_2
\le \sum_{j\in S_1}\sum_{k\in S_2}
\norm{f_j\sigma*f_k\sigma}_2
\le M_\eps^2\gamma(\lambda/2M_\eps)\theta(\eps)^2,
\end{equation}
where $\gamma(\lambda)\to 0$ as $\lambda\to\infty$
by Lemma~\ref{lemma:distantcaps}.
Therefore
\begin{equation}
\begin{split}
\norm{F\sigma*F\sigma}_2^2
&\le \norm{F_1\sigma*F_1\sigma}_2^2
+ \norm{F_2\sigma*F_2\sigma}_2^2
+C\norm{f}_2^2\norm{F_1\sigma*F_2\sigma}_2
\\
&\le \Sbest^4\norm{F_1}_2^4
+ \Sbest^4\norm{F_2}_2^4
+ M_\eps^2\gamma(\lambda/2M_\eps)\theta(\eps)^2.
\end{split}
\end{equation}
Since $F_1,F_2$ have disjoint supports,
$\norm{F_1}_2^2+\norm{F_2}_2^2\le\norm{f}_2^2 = 1$ and consequently
\begin{equation}
\norm{F_1}_2^4
+ \norm{F_2}_2^4
\le \max\big(\norm{F_1}_2^2, \norm{F_2}_2^2\big)\cdot
\big(\norm{F_1}_2^2+\norm{F_2}_2^2 \big)
\le (1-\eps^2)\cdot 1\le 1-\eps^2.
\end{equation}
Thus
\begin{equation}
\norm{F\sigma*F\sigma}_2^2
\le \Sbest^4(1-\eps^2)
+ M_\eps^2\gamma(\lambda/2M_\eps)\theta(\eps)^2.
\end{equation}

Therefore
\begin{equation}\begin{split}
(1-\delta)^2\Sbest^2\le \norm{f\sigma*f\sigma}_2
&\le \norm{F\sigma*F\sigma}_2
+ C\norm{f}_2\norm{f-F}_2
\\
&\le \norm{F\sigma*F\sigma}_2
+C\eps^3,
\end{split}\end{equation}
so by transitivity
\begin{equation}
(1-\delta)^4\Sbest^4
\le C\eps^3
+ \Sbest^4(1-\eps^2)
+ M_\eps^2\gamma(\lambda/2M_\eps)\theta(\eps)^2.
\end{equation}
Since $\gamma(t)\to 0$ as $t\to\infty$,
for all sufficiently small $\eps>0$
this implies an upper bound, which depends only on $\eps$, for $\lambda$,
as was to be proved.
\end{proof}

\section{Step 6C: Upper bounds for extremizing sequences}

Proposition~\ref{prop:normalization}
states that any nearly extremal function satisfies
appropriately scaled upper bounds relative to some cap.
It is convenient for the proof to first observe that
a superficially weaker statement implies the version stated.
\begin{lemma} \label{lemma:altnormalization}
There exists a function $\Theta:[1,\infty)\to(0,\infty)$
satisfying $\Theta(R)\to 0$ as $R\to\infty$ with the following property.
For any $\eps>0$ and $\bar R\in[1,\infty)$ there exists $\delta>0$
such that any nonnegative even function $f$ satisfying
$\norm{f}_2=1$ which is $\delta$--nearly extremal
may be decomposed as $f=F+G$ where
$F,G$ are even and nonnegative with disjoint supports,
$\norm{G}_2<\eps$,
and there exists a cap $\scriptc=\scriptc(z,r)$
such that for any $R\in[1,\bar R]$,
\begin{align} \label{Theta1}
&\int_{\min(|x-z|,|x+z|)\ge Rr}F^2(x)\,dx
\le\Theta(R),
\\
\label{Theta2}
&\int_{F(x)\ge Rr^{-1}}F^2(x)\,dx
\le\Theta(R).
\end{align}
\end{lemma}

\begin{proof}[Proof that Lemma~\ref{lemma:altnormalization}
implies Proposition~\ref{prop:normalization}]
Let $\Theta$ be the function promised by the lemma.
Let $\eps,f$ be given, and assume without loss of generality
that $\eps$ is small.
Assuming as we may that $\Theta$ is a continuous, strictly
decreasing function, define
$\bar R = \bar R(\eps)$ by the equation $\Theta(\bar R)=\eps^2/2$.
Let $\scriptc=\scriptc(z,r)$
and $\delta=\delta(\eps,\bar R(\eps))$ along with $F,G$
satisfy the conclusions of the lemma relative to $\eps,\bar R(\eps)$.
Define $\chi$ to be the characteristic function
of the set of all $x\in S^2$ which satisfy either
$\min(|x-z|,|x+z|)\ge \bar R r$,
or $F(x)>\bar R|\scriptc|^{-1/2}$.
Redecompose $f=\tilde F + \tilde G$
where $\tilde F = (1-\chi)F$ and $\tilde G = G + \chi F$.
Then $\norm{\tilde G}_2<2\eps$,
while $\tilde F$ satisfies the required inequalities.
For instance,
if $R\le \bar R$ then
$\int_{\tilde F(x)\ge R|\scriptc|^{-1/2}} \tilde F(x)^2\,dx
\le 
\int_{F(x)\ge R|\scriptc|^{-1/2}} F(x)^2\,dx
\le\Theta(R)$,
while the integrand vanishes if $R>\bar R$.
\end{proof}

\begin{proof}[Proof of Lemma~\ref{lemma:altnormalization}]
Let $\eta:[1,\infty)\to(0,\infty)$ be a function to be chosen below,
satisfying $\eta(t)\to 0$ as $t\to\infty$.
This function will not depend on the quantity $\bar R$.

Let $\bar R\ge 1$, $R\in[1,\bar R]$, and $\eps>0$ be given.
Let $\delta=\delta(\eps,\bar R)>0$ be a small quantity to be chosen below.
Let $0\le f\in\lt(S^2)$ be even and $\delta$--nearly extremal.
It is no loss of generality to normalize so that $\norm{f}_2=1$.

Let $\{f_\nu\}$ be the sequence of functions obtained by applying the decomposition
algorithm to $f$.
Choose $\delta=\delta(\eps)>0$ sufficiently small and $M=M(\eps)$ sufficiently large
to guarantee that $\norm{G_{M+1}}_2<\eps/2$
and that $f_\nu,G_\nu$ satisfy all conclusions of
Lemma~\ref{lemma:Galsosmall} and Lemma~\ref{lemma:decompuniformdecay}
for $\nu\le M$. Set $F = \sum_{\nu=0}^M f_\nu$.
Then $\norm{f-F}_2=\norm{G_{M+1}}_2<\eps/2$.

Let $N\in\{0,1,2,\cdots\}$ be the minimum of $M$,
and the smallest number such that $\norm{f_{N+1}}_2<\eta$.
$N$ is majorized by a quantity which depends only on $\eta$.
Set $\scriptf=\scriptf_N=\sum_{k=0}^N f_k$.
It follows from Lemma~\ref{lemma:Galsosmall}, part (ii), that
\begin{equation} \label{eq:fminusscriptFsmall}
\norm{F-\scriptf}_2<\gamma(\eta)
\text{ where $\gamma(\eta)\to 0$ as $\eta\to 0$.}
\end{equation}
This function $\gamma$ is independent of $\eps,\bar R$.

To prove the lemma,
we must produce an appropriate cap $\scriptc=\scriptc(z,r)$,
and must establish the existence of $\Theta$.
To do the former is simple:
To $f_0$ is associated a cap $\scriptc_0=\scriptc(z_0,r_0)$ such that
$f_0\le C|\scriptc_0|^{-1/2}(\chi_{\scriptc_0\cup -\scriptc_0})$.
$\scriptc=\scriptc_0$ is the required cap.
Note that by Lemma~\ref{lemma:mvv},
$\norm{f_0}\ge c$ for some positive universal constant $c$.

Suppose that functions $R\mapsto\eta(R)$ and $R\mapsto\Theta(R)$
are chosen so that
\begin{gather}
\label{eq:Thetarequirement1}
\eta(R)\to 0 \text{ as } R\to\infty
\\
\label{eq:Thetarequirement2}
\gamma(\eta(R))\le \Theta(R) \text{ for all $R$.}
\end{gather}
Then by \eqref{eq:fminusscriptFsmall},
$F-\scriptf$ already satisfies the desired inequalities in
$\lt(S^2)$, so
it suffices to show that $\scriptf(x)\equiv 0$
whenever $\min(|x-z|,|x+z|)>Rr_0$,
and that $\norm{\scriptf}_\infty\le R|\scriptc_0|^{-1/2}$.

Each summand satisfies
$f_k\le C(\eta)|\scriptc_k|^{-1/2}\chi_{\scriptc_k\cup -\scriptc_k}$
where $C(\eta)<\infty$ depends only on $\eta$,
and in particular, $f_k$ is supported in $\scriptc_k\cup-\scriptc_k$.
$\norm{f_k}_2\ge\eta$ for all $k\le N$, by definition of $N$.
Therefore by Lemma~\ref{lemma:controlgeometry},
there exists a function $\eta\mapsto\lambda(\eta)<\infty$,
such that
if $\delta$ is sufficiently small as a function of $\eta$ then
$\varrho(\scriptc_k,\scriptc_0)\le \lambda(\eta)$ for all $k\le N$.
This is needed for $\eta=\eta(R)$ for all $R$ in the compact set
$[1,\bar R]$, so such a $\delta$ may be chosen as a function
of $\bar R$ alone; conditions already imposed on $\delta$ above
make it a function of both $\eps,\bar R$.

In the region of all $x\in S^2$ satisfying
$\min(|x-z_0|,|x+z_0|)>Rr_0$,
either $f_k\equiv 0$, or $\scriptc_k$ has radius $\ge \tfrac14 Rr_0$,
or the center $z_k$ of $\scriptc_k$
satisfies $\max(|z_k-z_0|,|z_k+z_0|)\ge \tfrac14Rr_0$.
Choose a function $R\mapsto\eta(R)$ which tends to $0$
sufficiently slowly as $R\to\infty$ to ensure that $\lambda(\eta(R))\to\infty$
sufficiently slowly that
the latter
two cases would
contradict the inequality $\varrho(\scriptc_k,\scriptc_0)\le \lambda$,
and therefore cannot arise.
Then $\scriptf(x)\equiv 0$ when $\min(|x-z_0|,|x+z_0|)>Rr_0$.

With the function $\eta$ specified, $\Theta$ can be defined by
\begin{equation}\label{Thetadefn}\Theta(R)=\gamma(\eta(R)).\end{equation}
Then
\eqref{Theta1} holds for all $R\in[1,\bar R]$.

We claim next that
$\norm{\scriptf}_\infty<R|\scriptc_0|^{-1/2}$
if $R$ is sufficiently large as a function of $\eta$.
Indeed, because the summands $f_k$ have pairwise disjoint
supports, it suffices to control $\max_{k\le N}\norm{f_k}_\infty$.
Again, by Lemma~\ref{lemma:Galsosmall},
$\norm{f_k}_\infty\le C(\eta)|\scriptc_k|^{-1/2}$.
If $\eta(R)$ is chosen to tend to zero sufficiently slowly as $R\to\infty$
to ensure that $C(\eta(R))\lambda(\eta(R))<R$ for all $k\le N$,
then inequality \eqref{Theta2}
holds provided that $\Theta$ is defined by \eqref{Thetadefn}.

The final function $\eta$ must be chosen to tend to zero slowly enough
to satisfy the requirements of these proofs of both \eqref{Theta1} and \eqref{Theta2}.
\end{proof}

\section{Preliminaries for Step 7}

Let a sequence of functions $\{g_\nu\}\subset\lt(\reals^2)$ satisfy
$g_\nu\ge 0$,
$\norm{g_\nu}_2\to 1$,
\begin{align}
\label{spatiallocalization}
&\int_{|x|\ge R} g_\nu(x)^2\,dx \le \Theta(R),
\\
\label{higher}
&\int_{g_\nu(x)\ge R} g_\nu(x)^2\,dx\le\Theta(R),
\end{align}
where
$\Theta(R)\to 0 \text{ as } R\to\infty$
uniformly in $\nu$.
Thus $g_\nu$ is upper normalized with respect to the unit ball $\scriptb\subset\reals^2$.
This prevents $g_\nu^2$ from converging
weakly to a Dirac mass; it forces any weak limit of $g_\nu^2$ to
be absolutely continuous, and to satisfy the same inequality involving $\Theta$.
In the proof of our main theorem, this situation arises with
$g_\nu=\phi_{\scriptc_\nu}^*(F_\nu)$ where $\{f_\nu\}$ is an extremizing sequence
with a decomposition $f_\nu=F_\nu+G_\nu$ satisfying $\norm{G_\nu}_2\to 0$,
and $F_\nu$ is upper even-normalized relative to a cap $\scriptc_\nu$.

The following simple lemma is the only place in the analysis where the nonnegativity
of an extremizing sequence is used.
\begin{lemma}
If $\{g_\nu\}$ satisfies the hypotheses listed above, then
for any $A>0$ there exists $c>0$ such that for all $\nu$,
\begin{equation}
\int_{|\xi|\le A}|\widehat{g_\nu}(\xi)|^2\,d\xi\ge c.
\end{equation}
\end{lemma}

\begin{proof}
Let $g\in\lt(\reals^2)$ be a nonnegative
function which satisfies $\norm{g}_2=1$
and the inequalities \eqref{spatiallocalization},\eqref{higher}.
For $t>0$ let $\varphi_t(y)=e^{-t|y|^2/2}$.
Then
\begin{equation}
\int g\varphi_t\,dy
= (2\pi)^{-2}\int \widehat{g}(\xi)\widehat{\varphi_t}(\xi)\,d\xi
= (2\pi)^{-1}t^{-1}\int \widehat{g}(\xi)e^{-|\xi|^2/2t}\,d\xi.
\end{equation}
For any $R,\rho\ge 1$
let $S=\{y: |y|\le R \text{ and } g(y)\le\rho\}$.
Provided that $R,\rho$ are chosen to be sufficiently large
that $\Theta(R)+\Theta(\rho)\le\tfrac12$,
\begin{align*}
\int_{\reals^2} g\varphi_t\,dy
&\ge
e^{-tR^2/2}\int_{S} g(y)\,dy
\\
&\ge
e^{-tR^2/2}\rho^{-1}\int_{S} g^2(y)\,dy
\\
&=
e^{-tR^2/2}\rho^{-1}\big( \norm{g}_2^2- \int_{\reals^2\setminus S} g^2(y)\,dy\big)
\\
&\ge \tfrac12 e^{-tR^2/2}\rho^{-1}
\end{align*}
for any $t>0$.
On the other hand,
by Cauchy-Schwarz
\begin{align*}
\int_{|\xi|\ge A} |\widehat{g}(\xi)|\,t^{-1}e^{-|\xi|^2/2t}\,d\xi
&\le
\pi^{1/2} t^{-1}
\norm{\widehat{g}}_2
\big(\int_{r=A}^\infty e^{-r^2/t}2r\,dr\big)^{1/2}
\\
& =
\pi^{1/2} t^{-1}
\big(t \int_{s=A^2/t}^\infty e^{-s}\,ds\big)^{1/2}
\\
& =
\pi^{1/2}t^{-1/2}
e^{-A^2/2t}.
\end{align*}
Cauchy-Schwarz also gives
\begin{equation*}
\begin{aligned}
\int_{|\xi|\le A}
|\widehat{g}(\xi)|\,t^{-1} e^{-|\xi|^2/2t}\,d\xi
&\le
\big(\int_{|\xi|\le A}|\widehat{g}(\xi)|^2\,d\xi \big)^{1/2}
(2\pi)^{1/2}\big(
\int_0^\infty t^{-2}e^{-r^2/t}\,r\,dr
\big)^{1/2}
\\
&=
\pi^{1/2}
t^{-1/2}
\big(\int_{|\xi|\le A}|\widehat{g}(\xi)|^2\,d\xi \big)^{1/2}.
\end{aligned}
\end{equation*}
Therefore
\begin{align*}
\pi^{1/2}t^{-1/2}
\Big(\int_{|\xi|\le A} |\widehat{g}(\xi)|^2\,d\xi\Big)^{1/2}
&\ge
\int_{\reals^2}
\widehat{g}(\xi)t^{-1} e^{-|\xi|^2/2t}\,d\xi
- \int_{|\xi|\ge A} |\widehat{g}(\xi)|\,t^{-1} e^{-|\xi|^2/2t}\,d\xi
\\
&\ge \pi e^{-tR^2/2}\rho^{-1}
- \pi^{1/2} t^{-1/2} e^{-A^2/2t}.
\end{align*}
Now substitute $t = A^2/\gamma$  where $\gamma=\gamma(A) \ge 1$ to obtain
\begin{equation}
\pi^{1/2}\gamma^{1/2}A^{-1}
\big(\int_{|\xi|\le A} |\widehat{g}(\xi)|^2\,d\xi\big)^{1/2}
\ge
\pi e^{-A^2 R^2/2\gamma}\rho^{-1}
- \pi^{1/2} \gamma^{1/2} A^{-1} e^{-\gamma/2}.
\end{equation}
$R,\rho$ have already been fixed, independent of $A$.
As all three of these quantities remain fixed and $\gamma\to\infty$,
this last lower bound tends
to $\pi \rho^{-1} -0>0$.
Thus choosing $\gamma$ sufficiently large yields the desired lower bound.
\end{proof}

\begin{lemma} \label{lemma:fourierseparated}
Let $c_0>0$.
Let $\{g_\nu\}$ be any sequence of functions in $\lt(\reals^2)$
satisfying $\norm{g_\nu}_{\lt}=1$ and
$\int_{|\xi|\le 1}|\widehat{g_\nu}(\xi)|^2\,d\xi\ge c_0$.
Then either there exists a function $\theta:[1,\infty)\to(0,\infty)$
satisfying
\begin{equation}\theta(s)\to 0 \text{ as } s\to\infty\end{equation}
such that
\begin{equation}
\int_{|\xi|\ge s} |\widehat{g_\nu}(\xi)|^2\,d\xi
\le\theta(s)
\qquad\text{ for all } s\in[1,\infty) \text{ and all $\nu$,}
\end{equation}
or there exist a subsequence $\nu_k\to\infty$
and real constants $\delta>0$, $\eps_k>0$, and $S_k\ge s_k\ge 1$
such that
$s_k\to\infty$,
$\eps_k\to 0$,
$S_k = s_k^3$,
\begin{align}
&\int_{|\xi|\le s_k}
|\widehat{g_{\nu_k}}(\xi)|^2\,d\xi
\ge\delta
\\
&\int_{|\xi|\ge S_k}
|\widehat{g_{\nu_k}}(\xi)|^2\,d\xi
\ge\delta
\\
&\int_{s_k\le|\xi|\le S_k}
|\widehat{g_{\nu_k}}(\xi)|^2\,d\xi
<\eps_k.
\end{align}
\end{lemma}
In this lemma,
$\delta$ is permitted, in principle,
to depend on $\{g_\nu\}$,
and $\eps_k,s_k$ are permitted to depend on
$\{g_\nu\}$ and on $k$ in an arbitrary manner,
provided only that they satisfy the stated conditions.

\begin{proof}
Define a sequence $\rho_1,\rho_2,\cdots$
by $\rho_1=2$ and by induction, $\rho_{j+1}=\rho_j^{3}$.
If the conclusion does not hold, then after passing to a subsequence and
renumbering, we have
\begin{equation}
\int_{|\xi|\ge \rho_\nu}
|\widehat{g_{\nu}}(\xi)|^2\,d\xi
\ge\delta
\text{ for all } \nu.
\end{equation}
Consider a large $\nu$. Since
\begin{equation}
\sum_{j=1}^{\nu-1}
\int_{\rho_j\le |\xi|\le \rho_{j+1}}
|\widehat{g_{\nu}}(\xi)|^2\,d\xi
\le (2\pi)^2\norm{g_\nu}_2^2\le (2\pi)^2
\end{equation}
and there are $\nu-1$ summands,
there must exist $j(\nu)$ satisfying
\begin{equation}
\int_{\rho_j\le |\xi|\le \rho_{j+1}}
|\widehat{g_{\nu}}(\xi)|^2\,d\xi
\le C\nu^{-1}.
\end{equation}
It suffices to set $s_\nu=\rho_{j(\nu)}$,
$S_\nu=\rho_{j(\nu)+1}=s_\nu^3$,
and $\eps_\nu = C\nu^{-1}$.
\end{proof}

\section{Step 7: Precompactness after rescaling}

Let $\{f_\nu\}$ be an even nonnegative extremizing sequence,
uniformly upper even-normalized with respect to caps $\scriptc_\nu$.
Set $g_\nu=\phi_\nu^* (f_\nu)$,
where $\phi_\nu$ is the rescaling map associated to $\scriptc_\nu$.
Suppose that $r_\nu\to 0$.
If the first conclusion of Lemma~\ref{lemma:fourierseparated}
holds, then we obtain the conclusion of
part (i) of Proposition~\ref{prop:precompactness}. If not, then after passing
to a subsequence,
$\{g_\nu\}$ satisfies the conclusions of the second alternative
of Lemma~\ref{lemma:fourierseparated}.

Split
\begin{equation}
g_\nu=g_\nu^0+g_\nu^\infty + g_\nu^\flat
\end{equation}
where
\begin{gather}
\norm{g_\nu^0}_2\ge\delta,
\\
\norm{g_\nu^\infty}_2\ge\delta,
\\
\norm{g_\nu^\flat}_2<\eps_\nu,
\\
\widehat{g_\nu^0}(\xi) \text{ is supported where } |\xi|\le 2s_\nu,
\\
\widehat{g_\nu^\infty}(\xi) \text{ is supported where } |\xi|\ge \tfrac12 S_\nu,
\\
g_\nu^0,g_\nu^\infty \text{ are upper normalized with respect to } \scriptb,
\\
\eps_\nu\to 0 \text{ as } \nu\to\infty.
\end{gather}
Here $\delta>0$ is a certain constant independent of $\nu$,
and $\scriptb$ denotes the unit ball in $\reals^2$.
This splitting is accomplished
via an appropriate $C^\infty$ three term partition of unity
in the Fourier space $\reals^2_\xi$.

Write $\scriptc_\nu=\scriptc(z_\nu,r_\nu)$.
The above decomposition of $g_\nu=\phi_\nu^*(f_\nu)$
induces a corresponding decomposition
\begin{equation}
f_\nu = F_\nu^0+F_\nu^\infty + F_\nu^\flat
\end{equation}
where
all three summands are real-valued and even, and
for all sufficiently large $\nu$,
\begin{align}
&\text{$F_\nu^0,F_\nu^\infty,F_\nu^\flat$ are upper even-normalized
with respect to $\scriptc_\nu$,}
\\
&\norm{F_\nu^\flat}_2\to 0 \text{ as } \nu\to\infty,
\\
&\norm{F_\nu^0}_2\ge\delta/2,
\\
&\norm{F_\nu^\infty}_2\ge\delta/2,
\\
&\text{$F_\nu^0$ and $F_\nu^\infty$ are supported
in $\scriptc(z_\nu,\tfrac12)$.}
\end{align}
Moreover:
\begin{lemma} \label{lemma:Fnufourierbounds}
The decomposition $f_\nu=F_\nu^0+F_\nu^\infty+F_\nu^\flat$ may
be carried out so that the above conditions are satisfied,
and moreover, for certain constants $C,C_N<\infty$,
the summands
$F_\nu^0,F_\nu^\infty$ admit representations
\begin{equation}\label{F-Repres}
F_\nu^0(y) = \int_{H_\nu} a_\nu^{0,\pm}(\xi) e^{iy\cdot\xi}\,d\xi,
\qquad
F_\nu^\infty(y) = \int_{H_\nu} a_\nu^{\infty,\pm}(\xi) e^{iy\cdot\xi}\,d\xi
\end{equation}
where the representations with $+$ signs are valid for $y\in\scriptc_\nu$,
and those with minus signs $-$ signs are valid for $y\in-\scriptc_\nu$,
with Fourier coefficients $a_\nu^{\pm},a_\nu^{\infty,\pm}$ satisfying
\begin{alignat}{2}
\label{FsFT1}
\int_{r_\nu |\xi|\le  S_\nu/4} |{a_\nu^{\infty,\pm}}(\xi)|^2\,d\xi
&\le C S_\nu^{-1} \qquad&&\text{ for all } \nu
\\
\label{FsFT2}
\int_{r_\nu |\xi|\ge 4 s_\nu} |{a_\nu^{0,\pm}}(\xi)|^2\,d\xi
&\le C_N s_\nu^{-N} \qquad &&\text{ for all } \nu, \text{ for any } N<\infty
\end{alignat}
where
$a_\nu^{0,\pm}(-\xi)\equiv \overline{a_\nu^{0,\pm}(\xi)}$
and
$a_\nu^{\infty,\pm}(-\xi)\equiv \overline{a_\nu^{\infty,\pm}(\xi)}$.
\end{lemma}

Details of the routine proof of this lemma are left to the reader.
Note that orthogonal projection of $\scriptc_\nu$ to $H_\nu$
is a bijection between open subsets of $S^2$ and of $H_\nu^2$, which
may be identified with $\reals^2$.
Thus $e^{iy\cdot\xi}$ depends only on the projection of $y\in S^2$
onto $H_\nu$ in these expressions. Since $(y_1,y_2,-y_3)$
has the same projection as $(y_1,y_2,y_3)$, the two hemispheres of
$S^2$ require different, though related, representations.
One cannot simply employ a dilation of $H_\nu$
to convert the inverse Fourier transform representations of $g_\nu^0,g_\nu^\infty$
into the desired representations of $F_\nu^0,F_\nu^\infty$ respectively,
because the resulting $F_\nu^0,F_\nu^\infty$ would not have compact supports
when regarded as functions with domains $H_\nu$,
and hence could not be regarded as functions with domains $S^2$.
Therefore dilations of $g_\nu^0,g_\nu^\infty$ must be multiplied
by smooth cutoff functions, which depend on the centers $z_\nu$
of $\scriptc_\nu$ but not on the radii $r_\nu$.
This introduces extra terms, which are incorporated into $F_\nu^\flat$.
As $r_\nu\to 0$, these extra terms tend to zero in $\lt(S^2)$. The
other conclusions follow readily.
Because $F_\nu^0,F_\nu^\infty$ are even functions,
it may of course be arranged that $a_\nu^{0,-}(-\xi)\equiv a_\nu^{0,+}(\xi)$,
and likewise for $a_\nu^{\infty,\pm}$.


As $\nu\to\infty$,
\begin{equation}
\norm{f_\nu\sigma*f_\nu\sigma}_2
\le
\norm{ \big(F_\nu^0\sigma*F_\nu^0\sigma\big) +
\big(F_\nu^\infty\sigma*F_\nu^\infty\sigma\big)}_2
+ 2\norm{F_\nu^0\sigma*F_\nu^\infty\sigma}_2
+o(1)
\end{equation}
where $o(1)$ denotes a function which tends to zero as $\nu\to\infty$.
Applying the triangle inequality to the first term does not lead to
a useful bound. Instead,
\begin{equation*}
\begin{aligned}
\norm{
\big(F_\nu^0\sigma*F_\nu^0\sigma \big)&+
\big(F_\nu^\infty\sigma*F_\nu^\infty\sigma\big)}_2^2
\\
&\le
\norm{ F_\nu^0\sigma*F_\nu^0\sigma}_2^2
+ \norm{ F_\nu^\infty\sigma*F_\nu^\infty\sigma}_2^2
+ 2\big|\big\langle F_\nu^0\sigma*F_\nu^0\sigma,\
F_\nu^\infty\sigma*F_\nu^\infty\sigma\big\rangle\big|
\\
&=
\norm{ F_\nu^0\sigma*F_\nu^0\sigma}_2^2
+ \norm{ F_\nu^\infty\sigma*F_\nu^\infty\sigma}_2^2
+ 2\big|\big\langle F_\nu^0\sigma*F_\nu^\infty\sigma,\
F_\nu^0\sigma*F_\nu^\infty\sigma\big\rangle\big|
\end{aligned}
\end{equation*}
since $F_\nu^0,F_\nu^\infty$ are real and even.
Therefore, since $F_\nu^0,F_\nu^\infty$
have uniformly bounded $\lt$ norms,
\begin{equation} \label{precross}
\norm{f\sigma*f\sigma}_2^2
\le
\norm{ F_\nu^0\sigma*F_\nu^0\sigma}_2^2
+ \norm{ F_\nu^\infty\sigma*F_\nu^\infty\sigma}_2^2
+ C\norm{F_\nu^0\sigma*F_\nu^\infty\sigma}_2
+o(1).
\end{equation}

The following key lemma will be proved below.
\begin{lemma} \label{lemma:smallcrossterm}
Let $F_\nu^0,F_\nu^\infty$ be upper even-normalized
with respect to a sequence of caps of radii $\le \tfrac12$.
Assume that $F_\nu^0,F_\nu^\infty$
admit Fourier representations satisfying the inequalities
of Lemma~\ref{lemma:Fnufourierbounds}. Then
\begin{equation}
\norm{F_\nu^0\sigma*F_\nu^\infty\sigma}_{\lt(\reals^3)}\to 0.
\end{equation}
\end{lemma}

\begin{corollary}
The second alternative cannot hold in Lemma~\ref{lemma:fourierseparated}.
\end{corollary}

\begin{proof}
Assume Lemma~\ref{lemma:smallcrossterm}.
Then  by \eqref{precross},
\begin{align*}
\norm{f\sigma*f\sigma}_2^2
&\le
\norm{ F_\nu^0\sigma*F_\nu^0\sigma}_2^2
+ \norm{ F_\nu^\infty\sigma*F_\nu^\infty\sigma}_2^2
+o(1)
\\
&\le {\mathbf S}^4\norm{F_\nu^0}_2^4
+
{\mathbf S}^4\norm{F_\nu^\infty}_2^4
+o(1).
\end{align*}

Since $S_\nu/s_\nu\to\infty$ and $\norm{F_\nu^\flat}_2\to 0$, it follows easily from
\eqref{FsFT1},\eqref{FsFT2} that
\begin{equation}
\norm{F_\nu^0}_2^2+\norm{F_\nu^\infty}_2^2
\le (1+o(1))\norm{f_\nu}_2^2=1+o(1).
\end{equation}
Since $\min\big(\norm{F_\nu^0}_2,\norm{F_\nu^\infty}_2\big)\ge\delta/2$,
this forces
\begin{equation}
\max\big( \norm{F_\nu^0}_2^2,\,\, \norm{F_\nu^\infty}_2^2 \big) \le 1-\rho
\end{equation}
for all sufficiently large $\nu$, for some $\rho>0$
independent of $\nu$.
It follows that
\begin{align*}
{\mathbf S}^4
\norm{F_\nu^0}_{\lt(\sigma)}^4
+
{\mathbf S}^4
\norm{F_\nu^\infty}_{\lt(\sigma)}^4
&\le {\mathbf S}^4
\Big( \norm{F_\nu^0}_{\lt(\sigma)}^2 + \norm{F_\nu^\infty}_{\lt(\sigma)}^2 \Big)
\max\big( \norm{F_\nu^0}_2^2,\,\, \norm{F_\nu^\infty}_2^2 \big)
\\
&\le {\mathbf S}^4
(1+o(1))(1-\rho).
\end{align*}
We conclude that
\begin{equation}
\limsup_{\nu\to\infty}
\norm{f_\nu\sigma*f_\nu\sigma}_{\lt(\reals^3)}^2
< {\mathbf S}^4,
\end{equation}
contradicting the assumption that $\{f_\nu\}$ was an extremizing sequence.
\end{proof}

Combining the above results, the proof of Proposition~\ref{prop:precompactness},
in the case when $r_\nu\to 0$,
is complete modulo the proof of Lemma~\ref{lemma:smallcrossterm}.

\section{Step 8: Excluding small caps}

In this section we prove Proposition~\ref{prop:notsmall},
assuming the case $r_\nu\le \tfrac12$ of Proposition~\ref{prop:precompactness}.
Thus we need
to prove that the radii $r_\nu$ of the caps $\scriptc_\nu$
associated to an extremizing sequence $\{f_\nu\}$
of positive even functions cannot tend to zero.

\begin{lemma} \label{lemma:compareP}
Let $\{f_\nu\}$ be any sequence of real-valued, even functions on $S^2$
satisfying $\norm{f_\nu}_{\lt}=1$.
Suppose that $f_\nu$ is upper even-normalized with respect
to a cap $\scriptc_\nu=\scriptc(z_\nu,r_\nu)$, uniformly in $\nu$.
Suppose that the sequence of pullbacks $\phi_\nu^*(f_\nu)$
satisfies the first alternative in the conclusion of
Lemma~\ref{lemma:fourierseparated}.
Suppose that $r_\nu\to 0$.
Then there exists a sequence of functions $F_\nu:\paraboloid\to\reals$
satisfying $\norm{F_\nu}_2\to 1$
such that
\begin{equation}
\limsup_{\nu\to\infty}
\norm{F_\nu\sigma_P*F_\nu\sigma_P}_2
\ge (3/2)^{-1/2}
\limsup_{\nu\to\infty}
\norm{f_\nu\sigma*f_\nu\sigma}_2.
\end{equation}
\end{lemma}

\begin{proof}[Proof of Proposition~\ref{prop:notsmall}]
Let $\{f_\nu\}$ be an extremizing sequence of nonnegative even functions
for the inequality \eqref{secondSversion} satisfying $\norm{f_\nu}_2=1$.
There exists a sequence of caps $\scriptc_\nu=\scriptc(z_\nu,r_\nu)$
such that each $f_\nu$ is upper even-normalized with respect to $\scriptc_\nu$.
We must prove that $\liminf_{\nu\to\infty} r_\nu>0$.

If not, then by passing to a subsequence we may assume that $r_\nu \to 0$.
By Proposition~\ref{prop:precompactness}, the sequence of pullbacks
$g_\nu=\phi_\nu^*(f_\nu)$ is precompact in $\lt(\reals^2)$.
Thus the hypotheses of Lemma~\ref{lemma:compareP} are satisfied,
so there exists a sequence of functions $F_\nu\in\lt(\paraboloid)$
satisfying its conclusions.

Now
$\norm{F_\nu\sigma_P*F_\nu\sigma_P}_2
\le {\mathbf P}^2\norm{F_\nu}_{\lt(\paraboloid)}^2$
by the definition of ${\mathbf P}$.
Consequently
\begin{equation}
\limsup_{\nu\to\infty} \norm{f_\nu\sigma*f_\nu\sigma}_2
\le (3/2)^{1/2}{\mathbf P}^2.
\end{equation}
The left-hand side tends to ${\mathbf S}^2$ since
$\{f_\nu\}$ is an extremizing sequence for \eqref{secondSversion},
so ${\mathbf S}^2\le (3/2)^{1/2}{\mathbf P}^2$,
contradicting the inequality ${\mathbf S}\ge 2^{1/4}{\mathbf P}$
of Lemma~\ref{lemma:constantfunction}.
\end{proof}

\begin{proof}[Proof of Lemma~\ref{lemma:compareP}]
Write $\scriptc_\nu=\scriptc(z_\nu,r_\nu)$.
Decompose $2^{1/2}f_\nu(x)=\tilde f_\nu(x)+\tilde f_\nu(-x)
+ f_\nu^\flat(x)$
where $\tilde f_\nu$ is real,
$\tilde f_\nu$ is supported in $\scriptc(z_\nu,r_\nu^{1/2})$,
$\norm{f_\nu^\flat}_2\to 0$,
and the functions $\phi_\nu^*(\tilde f_\nu)$
satisfy the first alternative of the
conclusions of Lemma~\ref{lemma:fourierseparated}, uniformly in $\nu$.

Since $f_\nu$ is even and $\norm{f_\nu}_2=1$,
we have $\norm{\tilde f_\nu}_2\to 1$ as $\nu\to\infty$.
Moreover $g_\nu (x) = \tilde f_\nu(x) + \tilde f_\nu(-x)$ satisfies
\begin{equation}
\norm{g_\nu\sigma*g_\nu\sigma}_2^2/\norm{g_\nu}_2^4
\equiv
\tfrac32
\norm{\tilde f_\nu\sigma*\tilde f_\nu\sigma}_2^2/\norm{\tilde f_\nu}_2^4,
\end{equation}
and therefore
\begin{equation}
\lim_{\nu\to\infty}
\norm{\tilde f_\nu\sigma*\tilde f_\nu\sigma}_2^2
= (3/2)^{-1}
\lim_{\nu\to\infty}
\norm{f_\nu\sigma*f_\nu\sigma}_2^2.
\end{equation}

By rotation symmetry, we may suppose that $z_\nu=(0,0,1)$ for all $\nu$.
Define
$F_\nu:\paraboloid\to[0,\infty)$
by
\begin{equation}
F_\nu(y,|y|^2/2)
= r_\nu\tilde f_\nu\big(r_\nu y,(1-r_\nu^2|y|^2)^{1/2}\big)
\end{equation}
for $y\in\reals^2$.
$F_\nu$ will also be regarded as an element of $\lt(\reals^2,\,dy)$
by $F_\nu(y)=F_\nu(y,|y|^2/2)$.
Then $\norm{F_\nu}_{\lt(\paraboloid,\sigma_P)}
=\norm{F_\nu}_{\lt(\reals^2)}\to 1$ as $\nu\to\infty$.

It remains to prove that
\begin{equation}
\limsup_{\nu\to\infty}
\norm{\widehat{F_\nu\sigma_P}}_{L^4(\reals^3)}^4
\ge
\limsup_{\nu\to\infty}
\norm{\widehat{\tilde f_\nu\sigma}}_{L^4(\reals^3)}^4.
\end{equation}
We have
\begin{equation} \label{3uniform}
\int_{|y|\ge R}F_\nu(y)^2\,dy
+
\int_{F_\nu(y)\ge R}F_\nu(y)^2\,dy
+
\int_{|\xi|\ge R}|\widehat{F_\nu}(\xi)|^2\,d\xi\ \ \longrightarrow 0
\end{equation}
as $R\to\infty$, uniformly in $\nu$.


Thus we must compare
$\widehat{F_\nu\sigma_P}(x,t)=\int e^{-ix\cdot y-it|y|^2/2}F_\nu(y)\,dy$
with
\begin{equation}
\begin{aligned}
\widehat{\tilde f_\nu\sigma}(x,t)
&=\int_{\reals^2}
e^{-ix\cdot v-it(1-|v|^2)^{1/2}}\tilde f_\nu(v,(1-|v|^2)^{1/2})\,d\sigma(v,(1-|v|^2)^{1/2})
\\
&=\int_{\reals^2}
e^{-ix\cdot v-it(1-|v|^2)^{1/2}}\tilde f_\nu(v,(1-|v|^2)^{1/2})\,(1-|v|^2)^{-1/2}\,dv.
\end{aligned}
\end{equation}
In the latter integral substitute
$v=r_\nu y$ to obtain
\begin{align*}
r_\nu^{-1}
\widehat{\tilde f_\nu \sigma}&(r_\nu^{-1} x, - r_\nu^{-2} t)
\\
&=
r_\nu^{-1} r_\nu^2 \int_{\reals^2} e^{-ix\cdot y+itr_\nu^{-2}(1-r_\nu^2|y|^2)^{1/2}}
\tilde f_\nu(r_\nu y,(1-r_\nu^2|y|^2)^{1/2})\,(1-r_\nu^2|y|^2)^{-1/2}\,dy
\\
&=
\int_{\reals^2} e^{-ix\cdot y+itr_\nu^{-2}(1-r_\nu^2|y|^2)^{1/2}}F_\nu(y)
(1-r_\nu^2|y|^2)^{-1/2} \,dy
\\
&=
e^{itr_\nu^{-2}}
\int_{\reals^2} e^{-ix\cdot y-it|y|^2/2}
F_\nu(y)
h_\nu(t,y)
\,dy
\end{align*}
where
\begin{align*}
h_\nu(t,y)
&=
e^{it \psi_{\nu}(y)}
(1-r_\nu^2|y|^2)^{-1/2}
\\
\psi_\nu(y) &= -r_\nu^{-2}+|y|^2/2 + r_\nu^{-2} (1-r_\nu^2|y|^2)^{1/2}.
\end{align*}
Thus
\begin{align*}
\norm{\widehat{\tilde f_\nu\sigma}}_4^4
&=
\int_\reals\int_{\reals^2}
\big|r_\nu^{-1}\widehat{\tilde f_\nu\sigma}(r_\nu^{-1}x, - r_\nu^{-2}t)\big|^4\,dx\,dt
\\
&=
\Big\|
\int_{\reals^2} e^{-ix\cdot y-it|y|^2/2}
F_\nu(y)
h_\nu(t,y)
\,dy
\Big\|_{L^4(\reals^3)}^4.
\end{align*}
It will be important that
on any compact subset of $\reals^1_t\times\reals^2_{y}$,
\begin{equation} \label{htendsto1}
\text{$h_\nu(t,y)\to 1$ in the $C^N$ norm as $\nu\to\infty$, for all $N<\infty$.}
\end{equation}

Define
\begin{align}
u_\nu(x,t)
&=
\int_{\reals^2} e^{-ix\cdot y-it|y|^2/2}F_\nu(y)
\,h_\nu(t,y)\,dy
\\
\tilde u_\nu(x,t)
&=
\int e^{-ix\cdot y -it|y|^2/2}F_\nu(y)\,dy.
\end{align}

\begin{lemma} \label{lemma:uniformL4decay}
\begin{align}
&\int_{|(x,t)|\ge R}
|u_\nu(x,t)|^4\,dx\,dt\to 0
\text{ as } R\to\infty \text{ uniformly in } \nu.
\\
&\int_{|(x,t)|\ge R}|\tilde u_\nu(x,t)|^4\,dx\,dt
\to 0 \text{ as }R\to\infty \text{ uniformly in }\nu.
\end{align}
\end{lemma}

\begin{proof}
Define operators $T_\nu$ and $T$
from $\lt(\reals^2)$ to $L^4(\reals^3)$ by
\begin{align}
T_\nu g(x,t)
&=
\int_{\reals^2} e^{-ix\cdot y-it|y|^2/2}g(y)
\chi_{r_\nu^{-1}|y|\le 1/2}(y)
h_\nu(t,y)\,dy
\\
Tg(x,t)
&=
\int e^{-ix\cdot y -it|y|^2/2}g(y)\,dy.
\end{align}
$T:\lt(\reals^2)\to L^4(\reals^3)$ is bounded.
Although
the operators $T_\nu$ are written in coordinates which disguise this fact,
they are bounded
from $\lt(\reals^2)$ to $L^4(\reals^3)$
uniformly in $\nu$, because they are obtained via
norm-preserving changes of variables from the single bounded operator
$\lt(S^2,\sigma)\owns h\mapsto \widehat{h\sigma}$.

If $g\in C^2(\reals^2)$ has compact support, then
$|T_\nu g(x,t)|\le C_g |(x,t)|^{-1}$
where $C_g$ depends only on the $C^1$ norm of $g$
and on the diameter of its support,
provided that $\nu$ is sufficiently large that the support
of $g$ is contained in $B(0,r_\nu^{-1})$.
This follows from \eqref{htendsto1} together with
the method of stationary phase; the phase functions
appearing in the definition of $T_\nu$ have uniformly nondegenerate
critical points (if any), uniformly in $\nu$.

These two facts, together with the three uniform inequalities
\eqref{3uniform},
lead directly to the stated conclusion for $u_\nu$ by a routine argument.

A slightly simpler application of the same reasoning applies to
$\tilde u_\nu$.
\end{proof}

Therefore it suffices to prove that for any $R<\infty$,
\begin{equation} \label{utildeuclose}
\int_{|(x,t)|\le R}
\big|
u_\nu(x,t)-\tilde u_\nu(x,t)
\big|^4\,dx\,dt\to 0
\text{ as $\nu\to\infty$}.
\end{equation}
If $g\in L^1$ has compact support,
then
\begin{equation} \label{softconvergence}
|T_\nu(g)(x,t)-T(g)(x,t)|\to 0,
\text{ uniformly for all $|(x,t)|\le R$.}
\end{equation}
Since $T_\nu,T$ are uniformly bounded operators from
$\lt$ to $L^4$,
and since the class of all compactly supported $g\in L^1$
is dense in $\lt$,
\eqref{utildeuclose} follows from \eqref{softconvergence}.
\end{proof}

\section{Estimation of the cross term $\norm{F_\nu^0\sigma*F_\nu^\infty\sigma}_2^2$}

To prove Lemma~\ref{lemma:smallcrossterm},
let $f_\nu,F_\nu^0,F_\nu^\infty$
be as above. Let $f_\nu$ be upper even-normalized with respect to a cap $\scriptc_\nu$
of radius $r_\nu$. Since the inequality in question
is invariant under rotations of $\reals^3$,
we may suppose without loss of generality that
$\scriptc_\nu$ is centered at the north pole $z_0=(0,0,1)$.

Decompose $F_\nu^0 = F_\nu^{0,+}+ F_\nu^{0,-}$
where both summands are real-valued,
$F_\nu^{0,+}$ is supported in $\scriptc(z_0,\tfrac12)$,
$F_\nu^{0,-}(x)=F_\nu^{0,+}(-x)$,
$F_\nu^{0,\pm}$ is upper normalized with respect
to $\scriptc(\pm z_0,r_\nu)$,
and $F_\nu^{0,\pm}$ have the same Fourier representations
\eqref{F-Repres} as $F_\nu^0$.
There is a parallel decomposition $F_\nu^\infty = F_\nu^{\infty,+}+F_\nu^{\infty,-}$.
By Lemma~\ref{lemma:switchconvolutionfactors},
\begin{multline}
\norm{F_\nu^{0,+}\sigma*F_\nu^{\infty,+}\sigma}_2
=
\norm{F_\nu^{0,-}\sigma*F_\nu^{\infty,-}\sigma}_2
\\
=
\norm{F_\nu^{0,-}\sigma*F_\nu^{\infty,+}\sigma}_2
=
\norm{F_\nu^{0,+}\sigma*F_\nu^{\infty,-}\sigma}_2.
\end{multline}
Therefore it suffices to bound
$\norm{F_\nu^{0,+}\sigma*F_\nu^{\infty,+}\sigma}_2$.



\begin{lemma}
Let $\delta_\nu,\delta_\nu^*>0$ be sequences of positive numbers which satisfy
\begin{align}
\delta_\nu^{\phantom{*}}/r_\nu^2&\to 0 
\\
\delta_\nu^*/r_\nu^2&\to\infty. 
\end{align}
Then
\begin{equation}
\norm{F_\nu^{0,+}\sigma*F_\nu^{\infty,+}\sigma}_{\lt(\{x\in\reals^3: |x|>2-\delta_\nu
\text{ or } |x|<2-\delta_\nu^* \})}
\to 0 \text{ as } \nu\to\infty.
\end{equation}
\end{lemma}

\begin{proof}
Since $F_\nu^{0,+},F_\nu^{\infty,+}$
are upper normalized with respect to $\scriptc_\nu$,
Corollary~\ref{cor:neglectlargerx}
asserts that the region $|x|>2-\delta_\nu$
makes a small contribution for large $\nu$.
To handle the region $|x|<2-\delta_\nu^*$,
choose a sequence $t_\nu\ge 1$ tending slowly to infinity.
Decompose
$F_\nu^{0,+}=F_\nu^{0,+} \chi_{\scriptc(z_0, t_\nu r_\nu)}
+F_\nu^{0,+} \chi_{S^2\setminus \scriptc(z_0, t_\nu r_\nu)}$,
and decompose $F_\nu^{\infty,+}$ in the same way.
If $t_\nu\to\infty$ sufficiently slowly,
then the main term
$F_\nu^{0,+} \chi_{\scriptc(z_0, t_\nu r_\nu)}\sigma
*
F_\nu^{\infty,+} \chi_{\scriptc(z_0, t_\nu r_\nu)}\sigma$
is supported where $|x|>2-\delta_\nu^*$.
Expanding $F_\nu^{0,+}\sigma*F_\nu^{\infty,+}\sigma$ according
to this decomposition leaves three more terms.
Each of these has small norm in $\lt(\reals^3)$ for large $\nu$,
because $\norm{F_\nu^{0,+}}_{\lt(S^2\setminus\scriptc(z_0,t_\nu r_\nu))}\to 0$
and $\norm{F_\nu^{\infty,+}}_{\lt(S^2\setminus\scriptc(z_0,t_\nu r_\nu))}\to 0$.
\end{proof}

If $h_1,h_2$ are supported in $\scriptc(z_0,r)$
then $h_1\sigma*h_2\sigma$ is supported
in $\{x\in\reals^3: |x-2z_0|\le Cr\}$.
Since $F_\nu^{0,+},F_\nu^{\infty,+}$
are upper normalized with respect to $\scriptc(z_\nu,r_\nu)$,
and since $r_\nu\to 0$,
it follows from the inequality $\norm{h_1\sigma*h_2\sigma}_{\lt(\reals^3)}
\le C\norm{h_1}_2\norm{h_2}_2$ that
\begin{equation} \label{lastgasp}
\int_{|x-2z_0|\ge 1/100} |(F_\nu^{0,+}\sigma*F_\nu^{\infty,+}\sigma)(x)|^2\,dx
\to 0 \text{ as } \nu\to\infty.
\end{equation}
On the other hand, if $|x-2z_0|\le 1/100$,
then for all sufficiently large $\nu$,
$ (F_\nu^{0,+}\sigma*F_\nu^{\infty,+}\sigma)(x)$
depends only on the restrictions of $F_\nu^{0,+},F_\nu^{\infty,+}$
to $\scriptc(z_0,1/10)$.
This has the following significance
in terms of the Fourier representations \eqref{FsFT1},\eqref{FsFT2}
of Lemma~\ref{lemma:Fnufourierbounds}:
\begin{equation} \label{FsFT3}
F_\nu^{0,+}(x)=
\int_{r_\nu|\zeta|\le 4s_\nu}
e^{ix\zeta} a_\nu^{0,+}(\zeta)\,d\zeta
\ + \  o(1) \text{ in $\lt(\scriptc(z_0,1/10))$ as $\nu\to\infty$}
\end{equation}
by virtue of \eqref{FsFT2};
this does not follow for $\lt(S^2)$ because surface measure
on $S^2$ is not approximately equivalent to Lebesgue measure on $\{(x_1,x_2,0)\}$ near
the equator $\{x\in S^2: x_3=0\}$.
Likewise, by \eqref{FsFT1},
\begin{equation} \label{FsFT4}
F_\nu^{\infty,+}(x)=
\int_{r_\nu|\zeta|\ge S_\nu/4}
e^{ix\zeta} a_\nu^{\infty,+}(\zeta)\,d\zeta
\ + \  o(1) \text{ in $\lt(\scriptc(z_0,1/10))$ as $\nu\to\infty$}.
\end{equation}
Henceforth we simplify notation by writing $a_\nu^0$ in place of $a_\nu^{0,+}$
and $a_\nu^{\infty}$ in place of $a_\nu^{\infty,+}$,
and we will take these functions to be supported in  the sets
$r_\nu|\zeta|\le 4s_\nu$ and $r_\nu|\zeta|\ge S_\nu/4$, respectively.

Set $H=\{\xi\in\reals^3: \xi_3=0\}$,
and identify $(\xi_1,\xi_2,0)\in H$ with $(\xi_1,\xi_2)\in\reals^2$.
Denote by $\scripta_\nu$ the region
and $I_\nu$ the interval
\begin{align}
\scripta_\nu&=\{x\in\reals^3:
2-\delta_\nu^*\le|x|\le 2-\delta_\nu \text{ and } |x-2z_0|<1/100\}
\\
I_\nu&=[2-\delta_\nu^*,2-\delta_\nu].
\end{align}
It remains only to estimate
$\norm{F_\nu^{\infty,+}\sigma*F_\nu^{0,+}\sigma}_{\lt(\scripta_\nu)}$.
For $x\in \scripta_\nu$,
for all sufficiently large $\nu$,
$(F_\nu^{0,+}\sigma*F_\nu^{\infty,+}\sigma)(x)$
depends only on the restrictions of $F_\nu^{0,+},F_\nu^{\infty,+}$
to $\scriptc(z_0,1/10)$.
Therefore in majorizing
$\norm{F_\nu^{\infty,+}\sigma*F_\nu^{0,+}\sigma}_{\lt(\scripta_\nu)}$,
we may replace
$F_\nu^{0,+}(x)$ by
$\int_{r_\nu|\zeta|\le 4s_\nu} e^{ix\zeta} a_\nu^0(\zeta)\,d\zeta$
and
$F_\nu^{\infty,+}(x)$ by
$\int_{r_\nu|\zeta|\ge S_\nu/4} e^{ix\zeta} a_\nu^\infty(\zeta)\,d\zeta$,
at the expense of additional terms which are $o(1)$ as $\nu\to\infty$.
We will continue to denote these modified functions
by $F_\nu^{0,+},F_\nu^{\infty,+}$.

Set $h_\zeta= e_{-i\zeta} F_\nu^{\infty,+}$,
for $r_\nu|\zeta|\le 4s_\nu$.
Let
\begin{equation}
H^*=\{\zeta\in H: r_\nu|\zeta|\le 4s_\nu\}.
\end{equation}
By  \eqref{switchcharacter}, \eqref{Iversion}, \eqref{FsFT3}, and \eqref{FsFT4},
\begin{align*}
\norm{F_\nu^{\infty,+}\sigma*F_\nu^{0,+}\sigma}_{\lt(\scripta_\nu)}^2
&\le C
\int_{I_\nu}
\Big\| \int_{H^*} |a_\nu(\zeta)| \cdot|T_{\rho(t)} h_\zeta| \,d\zeta
\Big\|_{\lt(S^2)}^2 \,dt
\ +o(1)
\\
&\le C \int_{I_\nu}
\Big(
\int_{H^*} |a_\nu(\zeta)|\cdot
\norm{T_{\rho(t)} h_\zeta}_{\lt(S^2)}\,d\zeta
\Big)^2
\,dt
\ +o(1)
\\
&\le
C\norm{a_\nu}_{2}^2
\int_{H^*}
\int_{I_\nu}
\norm{T_{\rho(t)} h_\zeta}_{\lt(S^2)}^2\,dt\,d\zeta
\ +o(1)
\\
&\le C
\int_{H^*} \int_{I_\nu}
\norm{T_{\rho(t)} h_\zeta}_{\lt(S^2)}^2\,dt\,d\zeta
\ +o(1)
\end{align*}
by Minkowski's inequality and Cauchy-Schwarz.
Inserting the Fourier integral operator bound
$\norm{T_\rho(h_\zeta)}_2^2
\le C \norm{(I-\rho^2\Delta)^{-1/4}h_\zeta}_2^2$
yields
\begin{align}
\notag
\norm{F_\nu^{\infty,+}\sigma*F_\nu^{0,+}\sigma}_{\lt(\scripta_\nu)}^2
&\le C
\int_{\zeta\in H^*} \int_{I_\nu}
\int_{\xi\in H} (1+\rho(t)|\xi|)^{-1}|\widehat{h_\zeta(\xi)}|^2\,d\xi
\,dt\,d\zeta
\ +o(1)
\\
\notag
&= C
\int_{\zeta\in H^*} \int_{I_\nu}
\int_{\xi \in H} (1+\rho(t)|\xi|)^{-1}|{a_\nu^\infty}(\xi-\zeta)|^2\,d\xi
\,dt\,d\zeta
\ +o(1)
\\
\label{eq:willcontinue}
&\sim
s_\nu^2 r_\nu^{-2}
\int_{I_\nu}
\int_{H} (1+\rho(t)|\xi|)^{-1}|{a_\nu^\infty}(\xi)|^2\,d\xi
\,dt
\ +o(1)
\end{align}
since $|\xi|\gg|\zeta|$ for $\zeta$ in the support of $a_\nu^{0}$
and $\xi$ in the support of $a_\nu^\infty$.
Next,
\begin{align}
\notag
\int_{H}
(1+\rho|\xi|)^{-1}
&|{a_\nu^\infty}(\xi)|^2\,d\xi
\\
\label{applyFsFT1}
&\le
C \int_{r_\nu|\xi|\le c_0 S_\nu} |{a_\nu^\infty}(\xi)|^2\,d\xi
+
C \int_{r_\nu|\xi|\ge c_0 S_\nu}
(1+\rho|\xi|)^{-1}
|{a_\nu^\infty}(\xi)|^2\,d\xi
\\
\notag
&\le C S_\nu^{-1} \norm{F_\nu^{\infty,+}}_2^2
+ C\max_{r_\nu|\xi|\ge c_0 S_\nu} (1+\rho|\xi|)^{-1}
\cdot \norm{F_\nu^{\infty,+}}_2^2
\\
\notag
&\le C
S_\nu^{-1} + C\rho^{-1}r_\nu S_\nu^{-1}.
\end{align}
The first term in \eqref{applyFsFT1} was estimated using \eqref{FsFT1}.
Inserting the final line into \eqref{eq:willcontinue} yields
\begin{align*}
\norm{F_\nu^{\infty,+}\sigma*F_\nu^{0,+}\sigma}_{\lt(\scripta_\nu)}^2
&\le
C s_\nu^2 r_\nu^{-2}
\int_{I_\nu}
\big(S_\nu^{-1} + \rho(t)^{-1}r_\nu S_\nu^{-1}\big)
\,dt
\\
&\le
C s_\nu^2 r_\nu^{-2}
\int_{I_\nu}
\big(S_\nu^{-1} + (2-t)^{-1/2}r_\nu S_\nu^{-1}\big)
\,dt
\\
\intertext{since $(t/2)^2+\rho(t)^2=1$ implies $\rho(t)\ge C(2-t)^{1/2}$}
&= C s_\nu^2 S_\nu^{-1} r_\nu^{-2}
\int_{I_\nu}
(1+r_\nu(2-t)^{-1/2}) \,dt
\\
&\le C s_\nu^2 S_\nu^{-1} r_\nu^{-2}
|I_\nu|
\big(1+\max_{t\in I_\nu} r_\nu(2-t)^{-1/2}\big)
\\
&\le C s_\nu^2 S_\nu^{-1}
(r_\nu^{-2}\delta_\nu^*)
\big( 1+\delta_\nu^{-1/2}r_\nu \big)
\\
&\le C s_\nu^{-1}
(r_\nu^{-2}\delta_\nu^*)
\big( 1+\delta_\nu^{-1/2}r_\nu \big)
\end{align*}
since $S_\nu\ge s_\nu^3$.

Combining all terms, we have shown that
\begin{equation}
\norm{F_\nu^{0,+}\sigma*F_\nu^{\infty,+}\sigma}_2^2
\le o(1) +
C s_\nu^{-1} (r_\nu^{-2}\delta_\nu^*) \big( 1+\delta_\nu^{-1/2}r_\nu \big)
\end{equation}
as $\nu\to\infty$,
provided that $\delta_\nu/r_\nu^2\to 0$ and $\delta_\nu^*/r_\nu^2\to\infty$.
Since $s_\nu\to\infty$,
it is  possible to choose $\delta_\nu,\delta_\nu^*$ to satisfy
the additional constraint
\begin{equation}
s_\nu^{-1} (r_\nu^{-2}\delta_\nu^*) \big( 1+\delta_\nu^{-1/2}r_\nu \big)
\to 0 \text{ as } \nu\to\infty.
\end{equation}
With such a choice, we obtain
\begin{equation}
\norm{F_\nu^{0,+}\sigma*F_\nu^{\infty,+}\sigma}_2^2
\to 0 \text{ as } \nu\to\infty,
\end{equation}
completing the proof of Lemma~\ref{lemma:smallcrossterm}.
\qed

\section{Large caps}
It remains to prove Proposition~\ref{prop:precompactness}
in the case when $r_\nu > 1$.
Introduce a $C^\infty$ partition of unity
of $S^2$ by nonnegative even functions $\eta_j$,
each of which is supported in
$\scriptc(z_j,1/8)\cup\scriptc(-z_j,1/8)$ for some $z_j\in S^2$.
Decompose $f_\nu = \sum_j f_{\nu,j}$, where $f_{\nu,j}=\eta_j f_\nu$.
For each index $\nu$ we thus obtain the collection
of functions $g_{\nu,j}=\phi_\nu^*(f_{\nu,j})\in\lt(\reals^2)$.
Lemma~\ref{lemma:fourierseparated} is now modified in the natural way:
Either
there exists a function $\theta:[1,\infty)\to(0,\infty)$ satisfying
$\theta(s)\to 0 \text{ as } s\to\infty$
such that
\begin{equation}
\label{forlargecapcase}
\int_{|\xi|\ge s} |\widehat{g_{\nu,j}}(\xi)|^2\,d\xi
\le\theta(s)
\qquad\text{ for all } s\in[1,\infty) \text{ and all $\nu,j$,}
\end{equation}
or there exist $\delta,\eps_k,s_k,S_k$ as in that lemma,
such that the conclusions in the second case of that lemma
hold, with $|\widehat{g_\nu}|^2$ replaced by $\sum_j |\widehat{g_{\nu,j}}|^2$.
In the former case, the conclusion of Proposition~\ref{prop:precompactness}
is just a reformulation of the conjunction of \eqref{forlargecapcase}
with the upper normalization bounds for $f_\nu$.

It remains only to demonstrate that the latter case cannot arise.
If it did, then
by summing over $j$ one would obtain again a decomposition
$f_\nu = F_\nu^0 + F_\nu^\infty+F_\nu^\flat$
where $\lim_{\nu\to\infty}\norm{F_\nu^\flat}_2=0$,
$F_\nu^0$ is comparatively slowly varying,
and $F_\nu^\infty$ is highly oscillatory.
It would follow as above that
$\norm{F_\nu^0}_2^2+\norm{F_\nu^\infty}_2^2\to 1=\norm{f_\nu}_2^2$
and
$\norm{F_\nu^0\sigma*F_\nu^\infty\sigma}_{\lt(\reals^3)}\to 0$,
and then
that $\limsup_{\nu\to\infty}\norm{f_\nu\sigma*f_\nu\sigma}_2^2
<{\mathbf S}^4$, contradicting the assumption that $\{f_\nu\}$
is an extremizing sequence.
\qed

\section{Constants are local maxima}

Theorem~\ref{thm:localmax} asserts that constant functions
are local maxima.
Define
\begin{align}
\Psi(f) &=\norm{f\sigma*f\sigma}_{\lt(\reals^3)}^2
\\
\Phi(f) &= \frac{\Psi(f)}{\norm{f}_{\lt(S^2)}^4}.
\end{align}
Denote by $\one$ the constant function $\one(x)=1$ for all $x\in S^2$.

\begin{proof}[Proof of Theorem~\ref{thm:localmax}]
Since $\Phi(f)=\Phi(tf)$ for all $t>0$,
and since $\Phi(f)\le\Phi(|f|)$,
we may restrict attention to functions of the form
$f = \one+\eps g$ where $0\le\eps\le\delta$,
$g\perp\one$, $g$ is real-valued, and $\norm{g}_{\lt(S^2)}=1$.
We may further assume that $g(-x)=g(x)$, by Proposition~\ref{prop:symmetrize}.

$\one$ is a critical point for $\Phi$. Indeed, by rotation symmetry,
$f=\one$ satisfies the generalized Euler-Lagrange equation
$f = \lambda (f\sigma*f\sigma*f\sigma)\Big|_{S^2}$
which characterizes critical points.

A straightforward calculation gives the Taylor expansion
\begin{equation}
\Phi(\one+\eps g) = \Phi(\one) +
\eps^2\norm{\one}_{\lt(S^2)}^{-4}
\Big(
6\langle g\sigma *g\sigma,\,\sigma*\sigma\rangle
- 2\Psi(\one)\norm{\one}_{2}^{-2} \norm{g}_2^2
\Big)
+O(\eps^3)
\end{equation}
where $O(\eps^3)$ denotes a quantity whose absolute value is majorized
by $C\eps^3$, uniformly for $g\in\lt(S^2)$ satisfying $\norm{g}_2\le 1$.
Thus it suffices to show that
\begin{equation}
\sup_{\norm{g}_2=1}
6\langle g\sigma *g\sigma,\,\sigma*\sigma\rangle
< 2\Psi(\one)\norm{\one}_{2}^{-2}.
\end{equation}

The quantities $\Psi(\one)$ and $\norm{\one}_2$ can be
evaluated explicitly.
Firstly,
$\norm{\one}_2^2= \sigma(S^2) = 4\pi$.
Secondly,
\begin{equation}
(\sigma*\sigma)(x) = 2\pi|x|^{-1}\chi_{|x|\le 2}.
\end{equation}
Indeed, it follows from trigonometry that $\sigma*\sigma(x) = a|x|^{-1}\chi_{|x|\le 2}$
for some $a>0$,
and $a$ can be evaluated by
\begin{equation}
(4\pi)^2 =
\sigma(S^2)^2 = \int_{\reals^3} (\sigma*\sigma)(x)\,dx
= \int_0^2 ar^{-1} \cdot 4\pi r^2\,dr
= 8\pi a.
\end{equation}
Therefore
\begin{multline*}
\Psi(\one) = \int_{\reals^3} \big(\sigma*\sigma(x) \big)^2\,dx
= \int_{\reals^3} 4\pi^2|x|^{-2}\,dx
\\
= 4\pi^2\int_0^2 r^{-2}\cdot 4\pi r^2\,dr
= 4\pi^2\cdot 4\pi\cdot 2
=32\pi^3.
\end{multline*}
Therefore it suffices to prove that
\begin{equation}
\sup_{\norm{g}_2=1}
\langle g\sigma *g\sigma,\,\sigma*\sigma\rangle
<
\tfrac13\cdot 32\pi^3\cdot (4\pi)^{-1}
= \tfrac83 \pi^2
\end{equation}
where the supremum is taken over all real-valued, even $g\in\lt(S^2)$
satisfying $\norm{g}_2=1$ and $\int g\,d\sigma=0$.

The following key bound will be established below.
\begin{lemma}\label{lemma:spharmonics}
For all real-valued even functions $g\in\lt(S^2)$ satisfying $\int g\,d\sigma=0$,
\begin{equation}
\Big|\iint_{S^2\times S^2}g(x)g(y)|x-y|^{-1}\,d\sigma(x)\,d\sigma(y)\Big|
\le \tfrac45\pi\norm{g}_{\lt(S^2)}^2.
\end{equation}
\end{lemma}
The factor $\tfrac45\pi$ is optimal, and is attained if and only if
$g$ is a spherical harmonic of degree $2$.

Now for such $g$ satisfying $\norm{g}_2=1$,
\begin{align}
\langle g\sigma *g\sigma,\,\sigma*\sigma\rangle
&= \langle g\sigma*(\sigma*\sigma),g\rangle
\\
&= 2\pi\iint_{S^2\times S^2}g(x)g(y)|x-y|^{-1}\,d\sigma(x)\,d\sigma(y)
\\
&\le 2\pi\cdot \tfrac45\pi
= \tfrac85\pi^2
<\tfrac83\pi^2,
\end{align}
completing the proof of Theorem~\ref{thm:localmax}.
\end{proof}

\begin{proof}[Proof of Lemma~\ref{lemma:spharmonics}]
We first recall the Funk-Hecke Formula in the theory of spherical harmonics, see e.g.,  \cite[p. 29]{Muller:1998:special-functions} or \cite[Theorem A]{xu:2000:Funk_Hecke}.
\begin{theorem}[Funk-Hecke formula]
Let $d\ge 2$ and $k\ge 0$ be integers.
Let $f$ be a continuous function on $[-1,1]$ and $Y_k$ be a spherical harmonic of degree $k$, on the sphere $S^d$.
Then for any $x\in S^d$,
\begin{equation}\label{eq-5}
\int_{S^d} f(x\cdot y) Y_k(y)d\sigma(y)=\lambda_k Y_k(x),
\end{equation}
where $x\cdot y$ is the usual inner product in $\mathbb{R}^{d+1}$, and
\begin{equation*}
\lambda_k=\dfrac{\omega_d\int_{-1}^1 f(t)C_k^{\frac {d-1}{2}}(t)(1-t^2)^{\frac {d-2}{2}}dt}{C_k^{\frac {d-1}{2}}(1)\int_{-1}^1(1-t^2)^{\frac {d-2}{2}}dt},
\end{equation*}
where $\omega_d:=\frac {2\pi^{\frac {d+1}{2}}}{\Gamma(\frac {d+1}{2})}$ denotes the surface area of the unit sphere $S^d$, and $C_k^\nu(t)$ is the Gegenbauer polynomial defined by the generating function
\begin{equation}\label{eq-6}
(1-2rt+r^2)^{-\nu}=\sum_{k=0}^\infty C_k^{\nu}r^k,
\end{equation}
for $0\le r<1$ and $-1\le t\le 1$ and $\nu>0$.
\end{theorem}

For $\nu=1/2$ and $t=1$,
the generating formula becomes
$(1-r)^{-2/2} = \sum_{k=0}^\infty C_k^{1/2}r^k$,
so
\begin{equation}
C_k^{1/2}=1 \text{ for all } k\ge 0.
\end{equation}
For $d=2$, $(d-2)/2=0$ and $\omega_d = 4\pi$,
and the relevant index $\nu$ is $\nu = (d-1)/2 = 1/2$.
Therefore for $d=2$,
\begin{equation}\label{eq-7}
\lambda_k=2\pi \int_{-1}^1 f(t)C_k^{1/2}(t)dt.
\end{equation}

Choosing $\nu=1/2$ and set $r=1$ in the generating function \eqref{eq-6}, we obtain
$$(2-2t)^{-1/2}=\sum_{k=0}^\infty C_k^{1/2}(t).$$
This formula is not entirely valid, since \eqref{eq-6} only applies for $r<1$; but
all calculations below can be justified by writing the corresponding formulae for $r<1$,
and then passing to the limit $r=1$. We will omit these details, and work directly with $r=1$.

We also recall the following fact in \cite[Chapter 4, Corollary 2.16]{Stein-Weiss:1971:fourier-analysis}: for $S^2$, the polynomials $C_k^{1/2}(t)$, $k=0,1,\ldots,$ are mutually orthogonal with respect to the inner product
$\langle f, g\rangle =\int_{-1}^1 f(t)g(t)dt$.
So for $f=(2-2t)^{-1/2}$ in \eqref{eq-7} and for any $k\ge 0$, by orthogonality,
\begin{equation}\label{eq-8}
\begin{split}
\lambda_k=2\pi\int_{-1}^1 (2-2t)^{-1/2} C_k^{1/2}(t)dt
& =2\pi\int_{-1}^1 \sum_{m=0}^\infty C_m^{1/2}(t) C_k^{1/2}(t)dt
\\
&= 2\pi\int_{-1}^1 \bigl(C_k^{1/2}(t)\bigr)^2dt\\
&=\frac {4\pi}{2k+1},
\end{split}
\end{equation}where the last identity follows from the normalized value of $C_k^{1/2}(t)$ over $(-1,1)$, see e.g.,  \cite[p.461]{Andrews-Askey-Roy:1999:special-functions} or \cite[10.15, p.54]{Muller:1998:special-functions}.
Hence for $f(t)=(2-2t)^{-1/2}$, for $x\in S^2$,
\begin{equation}\label{eq-9}
\int_{S^2} f(x\cdot y) Y_k(y)d\sigma(y)=\frac {4\pi}{2k+1}Y_k(x),\, \forall \,k\ge 0.
\end{equation}

Now return to
$\iint g(x)g(y)|x-y|^{-1}\,d\sigma(x)\,d\sigma(y)$.
Here $|x-y|^{-1} =(2-2x\cdot y)^{-1/2} = f(x\cdot y)$
where $f(t) = (2-2t)^{-1/2}$.
Since all spherical harmonics of odd degrees are odd,
and since $g\perp\one$,
$g$ may be expanded as
$g = \sum_{k=1}^\infty Y_{2k}$
where each $Y_{2k}$ is a spherical harmonic of degree $2k$.
These are of course pairwise orthogonal in $\lt(S^2)$.
Therefore
\begin{multline}
\iint g(x)g(y)|x-y|^{-1}\,d\sigma(x)\,d\sigma(y)
= \sum_{k=1}^\infty \langle \lambda_{2k}Y_{2k},Y_{2k}\rangle
\\
= \sum_{k=1}^\infty \langle \frac{4\pi}{2(2k)+1}Y_{2k},Y_{2k}\rangle
\le \frac{4\pi}{5}\sum_{k=1}^\infty\norm{Y_{2k}}_2^2
= \frac{4\pi}{5}\norm{g}_2^2.
\end{multline}
This completes the proof of Lemma~\ref{lemma:spharmonics}.
\end{proof}

\medskip
\begin{remark}
Consider inequalities of the modified form
\begin{equation} \label{L44}
\int_{\reals^3} \big|(f\sigma*f\sigma)(x) \big|^2\,w(x)\,dx
\le C\norm{f}_{L^4(S^2)}^4,
\end{equation}
where $w\ge 0$ is any {radial} weight.
The modification consists in placing the $L^4$ norm on the right-hand side
of the inequality, instead of the $L^2$ norm.

If the inequality holds for some $C<\infty$, and if $w$ satisfies
$|\lambda_k(w)|\le \lambda_0(w)$
where
\[
\lambda_k(w) = 2\pi\int_{-1}^1 w((2+2t)^{1/2})(2+2t)^{-1/2}\,C_k^{1/2}(t)\,dt,
\]
then constant functions are (global) extremals.
This holds in particular for $w\equiv 1$.
\end{remark}

This is proved as follows, in the spirit of Foschi \cite{foschi}.
We may assume that $f\ge 0$.
\begin{align*}
\int_{\reals^3} (f\sigma*f\sigma)(x)^2\,w(x)\,dx
&\le
\int_{\reals^3} \big[(f^2\sigma*\sigma)(x)\big]^2 \,w(x)\,dx
\\
&=2\pi\iint_{S^2\times S^2} f^2(x)f^2(y)|x+y|^{-1}w(|x+y|)\,d\sigma(x)\,d\sigma(y).
\end{align*}
The first inequality follows from Cauchy-Schwarz,
and is an equality if $f$ is constant modulo null sets
on almost every circle (that is, the intersection of $S^2$ with an affine plane)
in $S^2$; thus if and only if $f$ is constant modulo $\sigma$--null sets.
Expand $f^2=\sum_{k=0}^\infty Y_k$ in spherical harmonics.
Then
\[
2\pi\iint_{S^2\times S^2} f^2(x)f^2(y)|x+y|^{-1}w(|x+y|)\,d\sigma(x)\,d\sigma(y)
=
2\pi\sum_{k=0}^\infty \lambda_k \norm{Y_k}_2^2
\le 2\pi\sup_k\lambda_k\norm{f}_4^4,
\]
for certain coefficients $\lambda_k$ which depend only on $w$.
If there is a valid inequality \eqref{L44} with $C<\infty$,
then $\lambda_0<\infty$.
Thus constant functions are extremizers.
If $\max_{k\ne 0}|\lambda_k(w)|<\lambda_0(w)$,
then $f$ is an extremizer if and only if $f^2$ has a spherical harmonic
expansion with $Y_k=0$ for all $k\ge 1$, that is,
if and only if $f^2$ is constant.
For $f\ge 0$, this forces $f$ to be constant.
\qed

\section{A variational calculation} \label{section:calculation}

Recall the notation $e_\xi(x)=e^{x\cdot\xi}$.  It is natural to study
$\norm{\widehat{f\sigma}}_4/\norm{f}_2$ for $f(x)=e_\xi(x)$, for several reasons.

\noindent
(i)
Extremizers for the paraboloid $\paraboloid=\{x: x_3=\tfrac12|x'|^2\}$
where $x'=(x_1,x_2)$ are Gaussian functions of $x'$; but these are simply
restrictions to $\paraboloid$ of simple exponentials $e^{x\cdot\xi}$
for $\xi\in\complex^3$ satisfying $\Re(\xi_3)<0$.
\newline
(ii) $(f\sigma*f\sigma)(x)$ is expressed for each $x$ as an
integral of a product of two factors. When $f=e_\xi$,
the integrand becomes a constant for each $x$,
and hence the Cauchy-Schwarz inequality becomes an equality when
applied to each such integral in an appropriate way.
Such equalities are the key to one proof \cite{foschi} that
Gaussians are extremal for $\paraboloid$.
\newline
(iii)
$\norm{e_\xi\sigma*e_\xi\sigma}_2/\norm{e_\xi}_2^2$
is susceptible to a perturbative analysis for large $|\xi|$.
\newline
(iv) This analysis appears more likely to be generalizable to
other manifolds than $S^2$,
than does the calculation of Lemma~\ref{lemma:constantfunction} for $f\equiv 1$.

For these reasons, we carry out in this section
a perturbative analysis of
$\norm{e_\xi\sigma*e_\xi\sigma}_2/\norm{e_\xi}_2^2$,
thereby establishing Proposition~\ref{prop:perturbative}.

We will work with functions  concentrated principally
in a very small neighborhood of the north pole $(0,0,1)$.
A point $z\approx (0,0,1)$ in $S^2$ can be written
as
\begin{equation}
(y,(1-|y|^2)^{1/2}) = (y,1-\tfrac12|y|^2 - \tfrac18|y|^4+O(|y|^6))
\end{equation}
where $y\in\reals^2$ and $|y|<1$.
Let $\sigma$ denote surface measure on $S^2$;
\begin{equation}
d\sigma= (1+\tfrac12 |y|^2+O(|y|^4))\,dy.
\end{equation}

For $z\in S^2$ and $\eps>0$ define
\begin{equation}
f_\eps(z) = \eps^{-1/2}
e^{(z_3-1)/\eps}\chi_{|(z_1,z_2)|<\tfrac12}\chi_{z_3>0}.
\end{equation}
Within the domain of $f_\eps$, the mapping
$(z_1,z_2,z_3)\leftrightarrow (z_1,z_2)$
is a one-to-one correspondence between $S^2$ and a ball in $\reals^2$.

$f_\eps$ is essentially $\eps^{-1/2}e^{-1/\eps}e_\xi$ where $\xi = (0,0,\eps^{-1})$;
the two functions differ by $O(e^{-c/\eps})$ in $\lt$ norm for some $c>0$.
The cutoff functions are inserted for convenience in the calculation.

For $(t,x)\in\reals^{1+2}$ define
\begin{equation}
u_\eps(t,x)
=
\int_{S^2} f_\eps(z)
e^{-i(x,t)\cdot z} \,d\sigma(z)
\end{equation}
where of course $(x,t)\cdot z = x_1z_1+x_2z_2+tz_3$.
Then
\begin{align*}
u_\eps(t,x)
& =
\eps^{-1/2}
\int_{S^2}
e^{(z_3-1)/\eps}
e^{-ix\cdot (z_1,z_2)} e^{-itz_3}
\, \tilde\chi(z)
\,d\sigma(z)
\\
\begin{split}
& =
\eps^{-1/2}
e^{-it}
\int_{\reals^2}
e^{\big(-\tfrac12|y|^2-\tfrac18|y|^4+O(|y|^6)\big)\eps^{-1}}
\\
&\qquad\qquad
e^{-ix\cdot y} e^{-it(-\tfrac12|y|^2-\tfrac18|y|^4+O(|y|^6))}
(1+\tfrac12 |y|^2+O(|y|^4))
\chi(y)\,dy
\end{split}
\end{align*}
where $\tilde\chi,\chi$ denote disks centered respectively at $(0,0,1)\in S^2$
and $0\in\reals^2$, which are independent of $\eps$.
A change of variables gives
\begin{multline*}
u_\eps(t,x) =
\eps^{1/2}
e^{-it}
\int_{\reals^2}
e^{-i\eps^{1/2}x\cdot y}
e^{ -(1-i\eps t)(\tfrac12|y|^2+\eps\tfrac18|y|^4+O(\eps^{-1}{|\eps^{1/2} y|^6)}) }
\\
(1+\tfrac12 \eps|y|^2+O(|\eps^{1/2} y|^4))
\chi(\eps^{1/2}y)\,dy.
\end{multline*}
Setting
\begin{align*}
v_\eps(t,x)
&= e^{-it/\eps}\eps^{-1/2}u_\eps(-\eps^{-1} t,\eps^{-1/2}x)
\\
&=
\int_{\reals^2}
e^{-i x\cdot y}
e^{ -(1+it)(\tfrac12|y|^2+\eps\tfrac18|y|^4+O(\eps^{-1}{|\eps^{1/2} y|^6)} }
(1+\tfrac12 \eps|y|^2+O(|\eps^{1/2} y|^4))
\chi(\eps^{1/2}y)\,dy
\end{align*}
we have
\begin{equation}\label{vuequality}
\norm{v_\eps}_{L^4(\reals^3)}^4
=
\norm{u_\eps}_{L^4(\reals^3)}^4.
\end{equation}

Set
\begin{equation}
w_\eps(t,x)
=
\int_{\reals^2}
e^{-i x\cdot y}
e^{ -(1+it)(\tfrac12|y|^2+\eps\tfrac18|y|^4)}
(1+\tfrac12 \eps|y|^2)
\,dy
\qquad\qquad\text{for }\eps\ge 0.
\end{equation}
Using the exact definition of $f_\eps$ rather than
the approximate expressions above, it is routine to verify that
\begin{equation}
\norm{w_\eps}_4^4 = \norm{v_\eps}_4^4 + O(\eps^2)
\text{ as } \eps\to 0^+.
\end{equation}
Since we are interested in first variations with respect to $\eps$
of the $L^4$ norm at $\eps=0$, it will suffice to analyze $\norm{w_\eps}_4^4$.
Also introduce
\begin{equation}
g_\eps(y) = e^{-\tfrac12|y|^2-\eps\tfrac18|y|^4}
\end{equation}
and
\begin{equation}
d\sigma_\eps(y) = (1+\eps\tfrac12|y|^2)\,dy.
\end{equation}
Then
\begin{equation}
\norm{f_\eps}_{\lt(\sigma)}^2=\norm{g_\eps}_{\lt(\sigma_\eps)}^2+O(\eps^2).
\end{equation}

Although $f_\eps$ is not well-defined in the limit $\eps=0$,
$\lim_{\eps\to 0^+} \norm{f_\eps}_2^2>0$ does exist,
and we will abuse notation by writing $\norm{f_0}_2^2$
to denote this quantity. We have
\begin{equation}
\norm{f_0}_2^2
= \int_{\reals^2} e^{-2|y|^2/2}\,dy.
\end{equation}
It is a routine exercise to verify that
$\eps\mapsto \norm{v_\eps}_4^4$
is a $C^\infty$ function on $[0,\infty)$;
hence the same goes for $\norm{w_\eps}_4^4$,
and for $\norm{u_\eps}_4^4$ by \eqref{vuequality}.
Similarly,
$\eps\mapsto\norm{f_\eps}_2^2$
is $C^\infty$ on $[0,\infty)$.

Consider the functional
\begin{equation}
\Psi(\eps) = \log\frac {\norm{u_\eps}_{L^4}^4} {\norm{f_\eps}_{\lt}^4},
\end{equation}
which is initially defined for $\eps>0$ but extends continuously
and differentiably to $\eps=0$.
Its derivative is
\begin{equation} \label{DPsiformula}
\partial_\eps\big|_{\eps=0}\Psi(\eps)
=
\frac
{\partial_\eps\norm{w_\eps}_4^4\big|_{\eps=0}}
{\norm{w_0}_4^4}
-
2\frac{\partial_\eps\big|_{\eps=0}\norm{g_\eps}_2^2}
{\norm{g_0}_2^2},
\end{equation}
and of course
\begin{equation}
\Psi(0) = \log(\scriptr_\paraboloid^4)
\end{equation}
where $\scriptr_\paraboloid$ \eqref{Rparaboloiddefn} is the optimal constant for the adjoint restriction
inequality for the paraboloid.

We will calculate:
\begin{lemma}
\begin{equation}
\frac{\partial\Psi}{\partial\eps}\bigg|_{\eps=0} >0.
\end{equation}
\end{lemma}
Proposition~\ref{prop:perturbative} follows, since by radial
symmetry, $\norm{e_\xi\sigma*e_\xi\sigma}_2/\norm{e_\xi}_2^2$
depends only on $|\xi|$.

The most involved calculation is that of the numerator in the first term
of \eqref{DPsiformula}.
To begin that calculation,
\begin{align*}
\partial_\eps\big|_{\eps=0} w_\eps(t,x)
&=
\int
\big[-\tfrac18(1+it) |y|^4 + \tfrac12|y|^2\big]
e^{-ix\cdot y} e^{-(1+it)|y|^2/2}\,dy
\\
&=
\Big[-\tfrac18(1+it) (-i/2)^{-2} \partial_t^2
+ \tfrac12 (-i/2)^{-1}\partial_t
\Big]
\int
e^{-ix\cdot y} e^{-(1+it)|y|^2/2}\,dy
\\
&=
\Big[\tfrac12(1+it) \partial_t^2 + i\partial_t \Big]
\int
e^{-ix\cdot y} e^{-(1+it)|y|^2/2}\,dy
\\
&=
\Big[\tfrac12(1+it) \partial_t^2 + i\partial_t \Big]
w_0(t,x)
\\
&=
\Big[\tfrac12(1+it) \partial_t^2 + i\partial_t \Big]
c_0(1+it)^{-1}e^{-|x|^2/2(1+it)}
\end{align*}
where
$c_0$ is a positive constant whose precise value will play no role,
since it will ultimately appear in both the numerator and denominator of a certain ratio.

Define
\begin{equation}
\phi(t,x) = -\tfrac12|x|^2(1+it)^{-1}-\log(1+it),
\end{equation}
so that
\begin{equation}w_0=c_0e^\phi.\end{equation}
The last quantity above may be written as
\begin{align*}
&=c_0\Big[\tfrac12(1+it) \partial_t^2 + i\partial_t \Big]
e^{\phi}
\\
&= \tfrac12 c_0(1+it) \big(\phi_t^2+\phi_{tt} \big) e^\phi
+ c_0i\phi_t e^\phi
\\
&=
\big( \tfrac12(1+it) (\phi_t^2+\phi_{tt}) + i\phi_t \big)
w_0
\end{align*}
where
$\phi_t,\phi_{tt}$ denote respectively the first and
second partial derivatives of $\phi$ with respect to $t$.

Now
\begin{align*}
\phi_t
&= \tfrac{i}2 |x|^2(1+it)^{-2}-i(1+it)^{-1}
\\
\phi_{tt}
&= \tfrac{i}2 (-2i) |x|^2(1+it)^{-3} -i(-i)(1+it)^{-2}
\\
&= |x|^2(1+it)^{-3}-(1+it)^{-2}
\\
\phi_t^2
&=
-\tfrac14|x|^4(1+it)^{-4}
+ |x|^2(1+it)^{-3}
-(1+it)^{-2}
\end{align*}
so
\begin{equation}
\phi_t^2+\phi_{tt}
= -\tfrac14|x|^4(1+it)^{-4}
+ 2|x|^2(1+it)^{-3}
-2(1+it)^{-2}.
\end{equation}
Consequently
\begin{multline}
\tfrac12(1+it) \big(\phi_t^2+\phi_{tt}\big)
+i\phi_t
\\
=
-\tfrac18 |x|^4(1+it)^{-3}
+ |x|^2(1+it)^{-2}
-(1+it)^{-1}
- \tfrac{1}2 |x|^2(1+it)^{-2}+ (1+it)^{-1}
\\
=
-\tfrac18|x|^4(1+t^2)^{-3}(1-it)^3
+ \tfrac12 |x|^2(1+t^2)^{-2}(1-it)^{2},
\end{multline}
whose real part is
\begin{multline}
\Re\Big(
\tfrac12 (1+it) \big(\phi_t^2+\phi_{tt}\big)
+i\phi_t
\Big)
\\
=
-\tfrac18|x|^4(1+t^2)^{-3}(1-3t^2)
+ \tfrac12|x|^2(1+t^2)^{-2}(1-t^2).
\end{multline}

Now
\begin{equation}
\partial_\eps\norm{w_\eps}_4^4
=
4\int
|w_\eps|^4
\Re\left( \frac{\partial_\eps w_\eps}{w_\eps}\right)
\end{equation}
and therefore
\begin{align*}
\partial_\eps\norm{w_\eps}_4^4\big|_{\eps=0}
&=
4 \iint
\Re\Big( \tfrac12 (1+it) \big(\phi_t^2+\phi_{tt}\big) +i\phi_t \Big)
|w_0(t,x)|^4\,dx\,dt
\\
&=c_0^4\int_{\reals}\int_{\reals^2}
\Big[ -\tfrac12|x|^4(1+t^2)^{-3}(1-3t^2)
+ 2|x|^2(1+t^2)^{-2}(1-t^2)
\Big]
\\&\qquad\qquad\qquad\qquad\qquad
(1+t^2)^{-2}|e^{-|x|^2/2(1+it)}|^4
\,dx\,dt
\\
&=
c_0^4\int_{\reals}\int_{\reals^2}
\Big[ -\tfrac12 |x|^4(1+t^2)^{-3}(1-3t^2)
+ 2|x|^2(1+t^2)^{-2}(1-t^2)
\Big]
\\&\qquad\qquad\qquad\qquad\qquad
(1+t^2)^{-2}\ e^{-2|x|^2/(1+t^2)} \ dx\,dt.
\end{align*}
Substituting $x = (1+t^2)^{1/2}\tilde x$
and then replacing $\tilde x$ by $x$ gives
\begin{equation*}
\partial_\eps\norm{w_\eps}_4^4\big|_{\eps=0}
=
c_0^4\int_{\reals}\int_{\reals^2}
\Big[ -\tfrac12 |x|^4(1-3t^2)
+ 2|x|^2(1-t^2)
\Big]
(1+t^2)^{-2}
e^{-2|x|^2} \,dx\,dt.
\end{equation*}

By substituting $x = 2^{-1/2}y$ in $\reals^2$
and then $r = s^{1/2}$ in $(0,\infty)$
we derive the identities
\begin{align*}
\int_{\reals^2}
e^{-2|x|^2}\,dx
&=\tfrac12 \int_{\reals^2} e^{-|y|^2}\,dy
= \pi\int_0^\infty e^{-r^2}\,r\,dr
= \tfrac12 \pi\int_0^\infty e^{-s}\,ds
= \frac\pi2
\\
\int_{\reals^2}
|x|^2
e^{-2|x|^2}\,dx
&= \frac\pi4\int_0^\infty se^{-s}\,ds
= \frac\pi4
\\
\int_{\reals^2}
|x|^4
e^{-2|x|^2}\,dx
&= \frac\pi8\int_0^\infty s^2e^{-s}\,ds
= \frac\pi4.
\end{align*}
Recall also that
\begin{alignat*}{2}
&\int_\reals (1+t^2)^{-1}\,dt &&= \pi
\\
&\int_\reals (1+t^2)^{-2}\,dt &&= \frac{\pi}{2}.
\end{alignat*}.

Using these formulas we obtain
\begin{align*}
\partial_\eps\norm{w_\eps}_4^4\big|_{\eps=0}
&=
c_0^4\int_{\reals}
\Big[ -\tfrac12 (1-3t^2)\frac\pi4
+ 2(1-t^2)\frac\pi4
\Big]
(1+t^2)^{-2}
\,dt
\\
&=
\tfrac\pi4 c_0^4 \int_\reals
(-\tfrac12 t^2+\tfrac32)
(1+t^2)^{-2}
\,dt
\\
&=
\tfrac\pi4 c_0^4 \int_\reals
\Big[
-\tfrac12(1+t^2)^{-1} +2(1+t^2)^{-2}
\Big]
\,dt
\\
&=
\tfrac\pi4 c_0^4
\big(
-\frac{\pi}2 + 2\frac{\pi}2
\big)
\\
&=
c_0^4 \frac{\pi^2}{8}.
\end{align*}

On the other hand,
\begin{align*}
\norm{w_0}_4^4
&= c_0^4 \int_\reals\int_{\reals^2}
(1+t^2)^{-2}
e^{-2|x|^2/(1+t^2)}\,dx\,dt
\\
&= c_0^4 \int_\reals\int_{\reals^2}
(1+t^2)^{-1}
e^{-2|y|^2}\,dy\,dt
\\
&=c_0^4 \tfrac12\pi^2.
\end{align*}
Therefore
\begin{equation}
\frac
{\partial_\eps\norm{w_\eps}_4^4\big|_{\eps=0}}
{\norm{w_0}_4^4}
=
\frac{\pi^2 c_0^4/8}{\pi^2 c_0^4/2 }
=\frac14.
\end{equation}

The variation of $\norm{g_\eps}_2^2$ must also be taken into account:
\begin{align*}
\partial_\eps
\int_{\reals^2} g_\eps(y)^2\,d\sigma_\eps(y)\ \Big|_{\eps=0}
&=
\partial_\eps
\int_{\reals^2}
e^{-|y|^2-\eps\tfrac14|y|^4}(1+\eps\tfrac12|y|^2)\,dy\ \Big|_{\eps=0}
\\
&=
\int_{\reals^2}
(-\tfrac14|y|^4+\tfrac12|y|^2)
e^{-|y|^2}\,dy
\\
&= -\frac{2\pi}4 + \frac{\pi}2
\\
&= 0.
\end{align*}
Therefore
\begin{equation}
2\frac{\partial_\eps\norm{g_\eps}_{\lt(\sigma_\eps)}^2\big|_{\eps=0}} {\norm{g_0}_2^2}
= 0.
\end{equation}
Putting it all together,
\begin{equation}
\partial_\eps \Psi(\eps) \big|_{\eps=0} = \tfrac14-0 >0.
\end{equation}

\section{Proof of Lemma~\ref{lemma:neckpain}}
\label{section:neckpain}

\begin{proof}[Proof of Lemma~\ref{lemma:neckpain}]
Suppose that $f=\chi_E$ is the characteristic function of a set $E$.
We will begin by showing that there exist  $C<\infty$ and
exponents $s,t>0$ such that for any set $E$ and any index $k$,
\begin{multline} \label{charsetbound}
\sum_j |\scriptc_k^j|^2
\big(
|\scriptc_k^j|^{-1}
\int_{\scriptc_k^j}
|\chi_E|^p
\big)^{4/p}
\\
\le C|E|^2
\cdot
\min\big(2^{-2k}|E|^{-1},2^{2k}|E|\big)^t
\cdot
\max_i
\left(\frac{|E\cap\scriptc_k^i|}{|E|+|\scriptc_k^i|}\right)^{s}.
\end{multline}
Indeed,
\begin{align*}
\sum_j |\scriptc_k^j|^2
\big(
|\scriptc_k^j|^{-1}
\int_{\scriptc_k^j}
\chi_E^p
\big)^{4/p}
&=
\sum_j
|\scriptc_k^j|^2 |E\cap \scriptc_k^j|^{4/p}|\scriptc_k^j|^{-4/p}
\\
&\le
\sum_j |E\cap\scriptc_k^j|
\cdot
\max_i \Big(|E\cap \scriptc_k^i|^{4/p-1}|\scriptc_k^i|^{2-4/p}\Big)
\\
&=|E|
\max_i \Big(|E\cap \scriptc_k^i|^{4/p-1}|\scriptc_k^i|^{2-4/p}\Big).
\end{align*}
The analysis now splits into two cases.
Note that $|\scriptc_k^j|\sim 2^{-2k}$ uniformly for all indices $j,k$.
If $2^{-2k}\ge|E|$ then
\begin{align*}
|E|
\max_i \Big(|E\cap \scriptc_k^i|^{4/p-1}|\scriptc_k^i|^{2-4/p}\Big)
&\le
|E|^2
\max_i
\left(\frac{|E\cap\scriptc_k^i|}{|\scriptc_k^i|}\right)^{4/p-2}
\\
&\le
|E|^2
(2^{2k}|E|)^{2/p-1}
\max_i
\left(\frac{|E\cap\scriptc_k^i|}{|\scriptc_k^i|}\right)^{2/p-1}.
\end{align*}
Since $1\le p<2$, $\frac2p-1>0$ and hence
this is a bound of the required form \eqref{charsetbound}.
When instead
$2^{-2k}<|E|$ then
since $4/p-1>1\ge\tfrac12$,
\begin{align*}
|E| \max_i \Big(|E\cap \scriptc_k^i|^{4/p-1}|\scriptc_k^i|^{2-4/p}\Big)
&=
|E|^2
(2^{2k}|E|)^{-1}
\max_i
\left(\frac{|E\cap\scriptc_k^i|}{|\scriptc_k^i|}\right)^{4/p-1}
\\
&\le
|E|^2
(2^{2k}|E|)^{-1}
\max_i
\left(\frac{|E\cap\scriptc_k^i|}{|\scriptc_k^i|}\right)^{1/2}
\\
&=
|E|^2
(2^{2k}|E|)^{-1/2}
\max_i
\left(\frac{|E\cap\scriptc_k^i|}{|E|}\right)^{1/2},
\end{align*}
which again is a bound of the desired form.
Thus \eqref{charsetbound} is proved.

Next consider a general function $f\in\lt(S^2)$.
By sacrificing a constant factor in the inequality,
we may assume that $f$ takes the form
$f = \sum_{\alpha=-\infty}^\infty 2^{\alpha}\chi_{E_\alpha}$
where the sets $E_\alpha$ are pairwise disjoint and
$|E_\alpha|<\infty$.
Invoking the preceding analysis for each summand
together with the triangle inequality for the
sum with respect to $\alpha$ yields
\begin{align}
\norm{f}_{X_p}^4
\label{beforepaininneck}
&\le C\sum_k
\Big(
\sum_\alpha 2^\alpha
|E_\alpha|^{1/2}
\cdot
\min\big(2^{-2k}|E_\alpha|^{-1},2^{2k}|E_\alpha|\big)^{t/4}
\cdot
\max_i
\left(\frac{|E_\alpha\cap\scriptc_k^i|}{|E_\alpha|+|\scriptc_k^i|}\right)^{s/4}
\Big)^4
\\
\label{paininneck}
&\le
C\left(\sum_\alpha 2^{4\alpha}|E_\alpha|^2
\max_{k,i}
\left(\frac{|E_\alpha\cap\scriptc_k^i|}{|E_\alpha|+|\scriptc_k^i|}\right)^{s}\right)^{1/2}
\norm{f}_2^2.
\end{align}

\eqref{paininneck} is deduced as follows.
For each integer $r$ define
\begin{gather}
a_r=\sum_{\beta: |E_\beta|\in[2^r,2^{r+1})}
2^\beta|E_\beta|^{1/2}
\max_{m,i}
\left(\frac{|E_\beta\cap\scriptc_m^i|}{|E_\beta|+|\scriptc_m^i|}\right)^{s/4}
\\
b_{k,r} =
\min\big(2^{-(r+2k)t/4},2^{(r+2k)t/4}\big).
\end{gather}
Then by \eqref{beforepaininneck},
\begin{multline}
\norm{f}_{X_p}
\le C
\Big(\sum_{k=0}^\infty
(\sum_{r=-\infty}^\infty
a_r b_{k,r})^4
\Big)^{1/4}
\\
\le C
\Big(\sum_{k=0}^\infty
(\sum_{r}a_r^4b_{k,r})
(\sum_r b_{k,r})^{3}
\Big)^{1/4}
\\
\le C
\Big(\sum_{k=0}^\infty
\sum_{r}a_r^4b_{k,r}
\Big)^{1/4}
\le C
(\sum_{r}^\infty a_r^4)^{1/4}.
\label{afterpaininneck}
\end{multline}
Finally for each $r$,
an application of H\"older's inequality with exponents $8,\tfrac87$ gives
\begin{align*}
a_r
&=
\sum_{\beta: |E_\beta|\sim 2^r} 2^\beta|E_\beta|^{1/2}
\max_{m,i}
\left(\frac{|E_\beta\cap\scriptc_m^i|}{|E_\beta|+|\scriptc_m^i|}\right)^{s/4}
\\
&\le
C2^{r/2}
\Big(
\sum_{\beta: |E_\beta|\sim 2^r} 2^{4\beta}
\max_{m,i}
\left(\frac{|E_\beta\cap\scriptc_m^i|}{|E_\beta|+|\scriptc_m^i|}\right)^{2s}
\Big)^{1/8}
\Big(
\sum_{\beta: |E_\beta|\sim 2^r} 2^{4\beta/7}
\Big)^{7/8}
\\
&\le C
\Big(
\sum_{\beta: |E_\beta|\sim 2^r} 2^{4\beta}|E_\beta|^2
\max_{m,i}
\left(\frac{|E_\beta\cap\scriptc_m^i|}{|E_\beta|+|\scriptc_m^i|}\right)^{s}
\Big)^{1/8}
\norm{f}_2^{1/2}
\end{align*}
since the sum of the finite series
$\sum_{\beta: |E_\beta|\sim 2^r} 2^{4\beta/7}$
is comparable to its largest term.

Continuing now from \eqref{afterpaininneck}, we have
\begin{align*}
\norm{f}_{X_p}^8\norm{f}_2^{-4}
&\le C
\sum_\alpha 2^{2\alpha}|E_\alpha|
\ \
\cdot
\ \
\sup_{\alpha}
2^{2\alpha}|E_\alpha|
\max_{k,i}
\left(\frac{|E_\alpha\cap\scriptc_k^i|}{|E_\alpha|+|\scriptc_k^i|}\right)^{s}
\\
&= C\norm{f}_2^4
\cdot
\sup_{\alpha}
\left(
\Big(2^{2\alpha}|E_\alpha|\norm{f}_2^{-2}\Big)
\max_{k,i}
\left(\frac{|E_\alpha\cap\scriptc_k^i|}{|E_\alpha|+|\scriptc_k^i|}\right)^{s}
\right)
\\
&\le C\norm{f}_2^4
\cdot
\sup_{\alpha}
\left(
\Big(2^{2\alpha}|E_\alpha|\norm{f}_2^{-2}\Big)^s
\max_{k,i}
\left(\frac{|E_\alpha\cap\scriptc_k^i|}{|E_\alpha|+|\scriptc_k^i|}\right)^{s}
\right)
\end{align*}
for some $0<s\le 1$.

It remains to show that
\begin{equation}
\sup_{\alpha}
\left(
\Big(2^{2\alpha}|E_\alpha|\norm{f}_2^{-2}\Big)
\max_{k,i}
\left(\frac{|E_\alpha\cap\scriptc_k^i|}{|E_\alpha|+|\scriptc_k^i|}\right)
\right)
\le C\sup_{m,j}\Lambda_{m,j}(f)^r
\end{equation}
for some positive exponent $r$.
Set
\begin{equation}
X=
\sup_{\alpha}
\left(
\Big(2^{2\alpha}|E_\alpha|\norm{f}_2^{-2}\Big)
\max_{k,i}
\left(\frac{|E_\alpha\cap\scriptc_k^i|}{|E_\alpha|+|\scriptc_k^i|}\right)
\right)
\end{equation}
Choose
an index $\alpha$ for which the supremum is attained up to a factor of at most $2$.
Then
\begin{equation}
\tfrac12 X \le
\Big(2^{2\alpha}|E_\alpha|\cdot\norm{f}_2^{-2}\Big)
\max_{k,i}
\left(\frac{|E_\alpha\cap\scriptc_k^i|}{|E_\alpha|+|\scriptc_k^i|}\right).
\end{equation}
The right-hand side is a product of two nonnegative factors, neither of which can exceed $1$,
so
\begin{equation}
2^{2\alpha}|E_\alpha|/\norm{f}_2^{2}\ge X/2
\text{ and there exist $k,i$ such that}
\frac{|E_\alpha\cap\scriptc_k^i|}{|E_\alpha|+|\scriptc_k^i|}\ge  X/4.
\end{equation}
Set $\scriptc=\scriptc_k^i$.
We have $|E_\alpha|\ge 2^{-2\alpha-1}X\norm{f}_2^2$,
and since $|E_\alpha\cap\scriptc| \le 2^{-\alpha}\int_\scriptc |f|$,
\begin{equation}
|\scriptc|^{-1}\int_\scriptc|f|
\ge
2^{\alpha}\frac{|E_\alpha\cap\scriptc|}{|\scriptc|}
\ge
2^{\alpha}\frac{|E_\alpha\cap\scriptc|}{|E_\alpha|+|\scriptc|}
\ge c2^{\alpha}X.
\end{equation}
Also
\begin{multline}
|\scriptc|^{-1}\int_\scriptc|f|
\ge
2^{\alpha}\frac{|E_\alpha\cap\scriptc|}{|E_\alpha|}
\cdot\frac{|E_\alpha|}{|\scriptc|}
\\
\ge
2^{\alpha}\frac{|E_\alpha\cap\scriptc|}{|E_\alpha|+|\scriptc|}
|\scriptc|^{-1}|E_\alpha|
\ge c2^{\alpha}X|\scriptc|^{-1}|E_\alpha|
\\
\ge c2^{\alpha}X|\scriptc|^{-1}
\cdot 2^{-2\alpha}\norm{f}_2^2 X
= c2^{-\alpha}\norm{f}_2^2 X^2.
\end{multline}
Taking the geometric mean of these two bounds yields
\begin{equation}
\frac{ |\scriptc|^{-1}\int_\scriptc|f|}
{|\scriptc|^{-1/2}\norm{f}_2}
\ge
c  X^{3/2},
\end{equation}
which by the definitions of $X$  and $\Lambda_{k,i}(f)$
is a bound of the desired form.
\end{proof}

\end{document}